\documentclass[11pt]{article}
\usepackage[tbtags]{amsmath}
\usepackage{amssymb}
\usepackage{amsthm}
\usepackage[misc]{ifsym}
\usepackage{cases}
\usepackage{mathrsfs}
\usepackage{graphicx}
\usepackage{xcolor}

\numberwithin{equation}{section}
\normalsize

\title{\bf Linear-Quadratic Stackelberg Mean Field Games and Teams with Arbitrary Population Sizes \thanks{This work is supported by National Key R\&D Program of China (2022YFA1006104), National Natural Science Foundations of China (12471419, 12271304), and Shandong Provincial Natural Science Foundations (ZR2024ZD35, ZR2022JQ01, ZR2022JQ31).}}

\author{\normalsize Wenyu Cong\thanks{\it School of Mathematics, Shandong University, Jinan 250100, P.R. China, E-mail: congwenyu@mail.sdu.edu.cn}, Jingtao Shi\thanks{\it Corresponding author. School of Mathematics, Shandong University, Jinan 250100, P.R. China, E-mail: shijingtao@sdu.edu.cn}, Bingchang Wang\thanks{\it School of Control Science and Engineering, Shandong University, Jinan 250061, P.R. China, E-mail: bcwang@sdu.edu.cn}}

\newtheorem{mythm}{Theorem}[section]
\newtheorem{mydef}{Definition}[section]

\newtheorem{Remark}{Remark}[section]
\newtheorem{example}{Example}[section]

\begin{document}

    \maketitle

    \noindent{\bf Abstract:}\quad This paper addresses a linear-quadratic Stackelberg mean field (MF) games and teams problem with arbitrary population sizes, where the game among the followers is further categorized into two types: non-cooperative and cooperative, and the number of followers can be finite or infinite.
    The leader commences by providing its strategy, and subsequently, each follower optimizes its individual cost or social cost.
    A new {\it de-aggregation method} is applied to solve the problem, which is instrumental in determining the optimal strategy of followers to the leader's strategy.
    Unlike previous studies that focus on MF games and social optima, and yield decentralized asymptotically optimal strategies relative to the centralized strategy set, the strategies presented here are exact decentralized optimal strategies relative to the decentralized strategy set.
    This distinction is crucial as it highlights a shift in the approach to MF systems, emphasizing the precision and direct applicability of the strategies to the decentralized context.
    In the wake of the implementation of followers' strategies, the leader is confronted with an optimal control problem driven by high-dimensional forward-backward stochastic differential equations (FBSDEs).
    By variational analysis, we obtain the decentralized strategy for the leader.
    By applying the de-aggregation method and employing dimension expansion to decouple the high-dimensional FBSDEs, we are able to derive a set of decentralized Stackelberg-Nash or Stackelberg-team equilibrium solution for all players.

    \vspace{2mm}

    \noindent{\bf Keywords:}\quad Stackelberg game, de-aggregation method, linear-quadratic stochastic optimal control, mean field games, social optima, decentralized control

    \vspace{2mm}

    \noindent{\bf Mathematics Subject Classification:}\quad 93E20, 60H10, 49K45, 49N70, 91A23

    \section{Introduction}

    {\it MF games} have attracted significant academic interest and are being applied across various domains,  such as system control, applied mathematics, and economics (\cite{Bensoussan-Frehse-Yam-13}, \cite{Gomes-Saude-14}, \cite{Caines-Huang-Malhame-17}, \cite{Carmona-Delarue-18}).
    The MF games theory serves as a framework for describing the behavior of models characterized by a large population, where the influence of the overall population is significant, despite the negligible impact on individual entities.
    The theoretical framework of MF games, pioneered by Lasry and Lions \cite{Lasry-Lions-07} and independently by Huang et al. \cite{Huang-Malhame-Caines-06}, have proved to be highly effective and tractable for analyzing stochastic controlled systems that are weakly coupled through MF interactions.
    Specifically, within the {\it linear-quadratic} (LQ) framework, MF games provide a flexible modeling apparatus applicable to a broad spectrum of real-world issues.
    The solutions yielded by LQ-MF games demonstrate significant and graceful characteristics.
    Current scholarly discourse has extensively explored MF games, especially within the LQ framework (\cite{Li-Zhang-08}, \cite{Bensoussan-Sung-Yam-Yung-16}, \cite{Moon-Basar-17}, \cite{Huang-Zhou-20}, \cite{Cong-Shi-24b}).
    Huang et al. \cite{Huang-Caines-Malhame-07} conducted research on $\varepsilon$-Nash equilibrium strategies in the context of LQ-MF games with discounted costs, building upon the {\it Nash certainty equivalence} (NCE) approach.
    Later, the NCE approach was extended to cases involving long-term average costs, as detailed in Li and Zhang \cite{Li-Zhang-08}.
    In the domain of MF games with major players, Huang \cite{Huang-10} investigated continuous-time LQ games, providing insights into $\varepsilon$-Nash equilibrium strategies.
    Huang et al. \cite{Huang-Wang-Wu-16} introduced a backward-major and forward-minor setup for an LQ-MF games, and decentralized $\varepsilon$-Nash equilibrium strategies for major and minor agents were obtained.
    Huang et al. \cite{Huang-Wang-Wu-16b} analyzed the backward LQ-MF games of weakly coupled stochastic large population systems under both full and partial information scenarios.
    Huang and Li \cite{Huang-Li-18} explored an LQ-MF games related to a class of stochastic delayed systems.
    Xu and Zhang \cite{Xu-Zhang-20} examined a general LQ-MF games for stochastic large population systems, where the individual diffusion coefficient depends on the state and control of the agent.
    Bensoussan et al. \cite{Bensoussan-Feng-Huang-21} considered an LQ-MF games with partial observation and common noise.

    In addition to noncooperative games, the concept of social optima within MF models has garnered significant attention.
    Social optimum control involves all participants collaborating to minimize the collective social cost, which is the aggregate of individual costs.
    This approach is characteristic of a team decision-making problem, as referenced in \cite{Ho-80}.
    Huang et al. \cite{Huang-Caines-Malhame-12} explored the social optima in the context of LQ-MF control and derived an asymptotic solution suitable for a team-optimal scenario.
    Further, Huang and Nguyen \cite{Huang-Nguyen-16} designed socially optimal strategies by analyzing FBSDEs.
    Arabneydi and Mahajan \cite{Arabneydi-Mahajan-15} delved into team-optimal control strategies under conditions of finite population and partial information.

    The Stackelberg differential game, alternatively referred to as the leader-follower differential game, emerges in markets where some firms have the upper hand and can exert more influence over others.
    In light of this dynamic, Stackelberg \cite{Stackelberg-1952} first proposed the idea of a hierarchical solution. Within this framework, there are two participants with distinct roles, one designated as the leader and the other as the follower.
    To attain a pair of Stackelberg equilibrium solutions, the differential game is generally divided into two phases.
    The first phase involves addressing the follower's problem.
    Initially, the leader publicly announces their strategy, transforming the two-player differential game into the single-player optimal control problem for the follower.
    The follower then promptly chooses an optimal strategy in response to the leader's declared strategy, aiming to optimize their own cost function.
    The second phase is where the leader selects an optimal strategy, presupposing that the follower will opt for their own optimal strategy, also with the goal of minimizing (or maximizing) their cost function.
    This forms another optimal control problem, this time for the leader.
    In summary, decision-making must be jointly accomplished by both players.
    Due to the inherent role asymmetry, one player must be subordinate to the other, necessitating that one player makes their decision after the other has concluded their decision-making process.
    The Stackelberg differential game is highly relevant in financial and economic practices, and it has attracted growing attention in applied research.
    Bagchi and Ba\c{s}ar \cite{Bagchi-Basar-81} explored an LQ Stackelberg stochastic differential game where the diffusion coefficient in the state equation does not involve state and control variables.
    This study laid the groundwork for understanding the strategic interactions under uncertainty within a leader-follower framework.
    Yong \cite{Yong-02} delved into a more generalized framework of LQ leader-follower differential game problems.
    In this study, coefficients of the state system and cost functional are stochastic, the diffusion coefficient in the state equation includes control variables, and the weight matrix in front of the control variables in the cost functional is not necessarily positive definite.
    Bensoussan et al. \cite{Bensoussan-Chen-Sethi-15} introduced several solution concepts based on players' information sets and investigated LQ Stackelberg differential games under adaptive open-loop and closed-loop memoryless information structures, where control variables do not enter the diffusion coefficient in the state equation.
    Shi et al. \cite{Shi-Wang-Xiong-16} is concerned with a leader-follower stochastic differential game with asymmetric information.
    Zheng and Shi \cite{Zheng-Shi-20} investigated a Stackelberg game involving {\it backward stochastic differential equations} (BSDEs) (Pardoux and Peng \cite{Pardoux-Peng-90}, Ma and Yong \cite{Ma-Yong-99}).
    Feng et al. \cite{Feng-Hu-Huang-22} examined a Stackelberg game associated with BSDE featuring constraints.
    Sun et al. \cite{Sun-Wang-Wen-23} conducted research on a zero-sum LQ Stackelberg stochastic differential game.
    Xiang and Shi \cite{Xiang-Shi-24} concerned with a two-person zero-sum indefinite stochastic LQ Stackelberg differential game with asymmetric informational uncertainties, where both the leader and follower face different and unknown disturbances.

    {\it Stackelberg MF games}, distinct from Stackelberg stochastic differential games of MF type (incorporating the expected values of state and control variables, as seen in references \cite{Du-Wu-19}, \cite{Lin-Jiang-Zhang-19}, \cite{Wang-Zhang-20}, \cite{Moon-Yang-21}, \cite{Lin-Zhang-23}, etc.), have been increasingly capturing the attention of researchers.
    Nourian et al. \cite{Nourian-Caines-Malhami-Huang-12} studied a large population LQ leader-follower stochastic multi-agent systems and established their $(\epsilon_1, \epsilon_2)$-Stackelberg-Nash equilibrium.
    Bensoussan et al. \cite{Bensoussan-Chau-Yam-15} and Bensoussan et al. \cite{Bensoussan-Chau-Lai-Yam-17} investigated Stackelberg MF games featuring delayed responses.
    Wang and Zhang \cite{Wang-Zhang-14} examined hierarchical games for multi-agent systems involving a leader and a large number of followers with infinite horizon tracking-type costs.
    Moon and Ba\c{s}ar \cite{Moon-Basar-18} considered the LQ Stackelberg MF games with the adapted open-loop information structure, and derived $(\epsilon_1, \epsilon_2)$-Stackelberg-Nash equilibrium.
    Yang and Huang \cite{Yang-Huang-21} conducted a study on LQ Stackelberg MF games involving a major player (leader) and $N$ minor players (followers).
    Si and Wu \cite{Si-Wu-21} explored a backward-forward LQ Stackelberg MF games, where the leader's state equation is backward, and the followers' state equation is forward.
    A static output feedback strategy for robust incentive Stackelberg games with a large population for MF stochastic systems was investigated in Mukaidani et al. \cite{Mukaidani-Irie-Xu-Zhuang-22}.
    Dayanikli and Lauri\`{e}re \cite{Dayanikli-Lauriere-23} proposed a numerical approach of machine learning techniques to solve Stackelberg problems between a principal and a MF of agents.
    Wang \cite{Wang-24} employed a direct method to solve LQ Stackelberg MF games with a leader and a substantial number of followers.
    Cong and Shi \cite{Cong-Shi-24a} delve into backward-forward stochastic systems LQ Stackelberg MF games by direct approach.
    Si and Shi \cite{Si-Shi-24} concerned with an LQ Stackelberg MF games with partial information and common noise.

    Conventionally, two approaches are employed in the resolution of MF games.
    One is termed the fixed-point approach (or top-down approach, NCE approach, see \cite{Huang-Caines-Malhame-07}, \cite{Huang-Malhame-Caines-06}, \cite{Li-Zhang-08}, \cite{Bensoussan-Frehse-Yam-13}, \cite{Carmona-Delarue-18}), which initiates the process by employing MF approximation and formulating a fixed-point equation.
    By tackling the fixed-point equation and scrutinizing the optimal response of a representative player, decentralized strategies can be formulated.
    The alternative approach is known as the direct approach (or bottom-up approach, refer to \cite{Lasry-Lions-07}, \cite{Wang-Zhang-Zhang-20}, \cite{Huang-Zhou-20}, \cite{Wang-Zhang-Zhang-22}, \cite{Wang-24}, \cite{Cong-Shi-24a}, \cite{Cong-Shi-24b}).
    This method commences by formally solving an $N$-player game problem within a vast and finite population setting.
    Subsequently, by decoupling or reducing high-dimensional systems, centralized control can be explicitly derived, contingent on the state of a specific player and the average state of the population.
    As the population size $N$ approaches infinity, the construction of decentralized strategies becomes feasible.
    In \cite{Huang-Zhou-20}, the authors addressed the connection and difference of these two routes in an LQ setting.
    However, whether employing the fixed-point method or the direct method, the resulting decentralized strategies are asymptotic.
    When the number of participants, $N$, is a finite large number, the conclusions drawn from the aforementioned methods can become significantly erroneous, thus becoming ineffective.
    A new method has been introduced to address this issue, see \cite{Wang-Zhang-Zhang-22}, \cite{Wang-Zhang-Fu-Liang-23} and \cite{Liang-Wang-Zhang-24}, taking the conditional expectation with respect to the information filter adapted to the $i$th agent's decentralized control set and de-aggregating the MF term to obtain a decentralized strategy.
    The set of decentralized strategies is an exact Nash equilibrium with respect to decentralized control set, and is applicable for arbitrary number of agents.
    We refer to this method as {\it de-aggregation method}.

    In this paper, we explore LQ Stackelberg MF games and teams problem.
    The leader initiates the process by disclosing their strategy, following which each follower optimizes its individual cost.
    Employing the de-aggregation method, we formulate a decentralized Stackelberg-Nash equilibrium and a decentralized Stackelberg-team equilibrium.
    With the leader's strategy given, we first address MF games or teams for followers by the variational analysis, resulting in a system of FBSDEs.
    Due to accessible in formation restriction, the decentralized strategies of followers are given by the conditional expectation of costates.
    By using the de-aggregation method, we represent the MF term $x^{(N)}$ in the form of linear combination of the $i$ agent's state $x_i$ and its expectation.
    Then, we represent the conditional expectation of costates as the form of linear combination of the $i$ agent's state $x_i$ and its expectation, and by comparing coefficients, we obtain the decentralized optimal strategies and Riccati equations for followers.
    Subsequent to the followers implementing their strategies, the leader encounters an optimal control problem.
    By variational analysis, we obtain the decentralized strategy for the leader.
    By decoupling a high-dimensional FBSDE with the de-aggregation method and dimension expansion, we construct a set of decentralized Stackelberg-Nash or Stackelberg-team equilibrium strategies.

    The main contributions of the paper are outlined as follows.

    $\bullet$ We summarize and propose the {\it de-aggregation method} for MF games, which is distinct from the traditional fixed-point method and direct method. This new approach offers an alternative way to handle the complexities inherent in MF games theory.
    Unlike fixed-point method and direct method, which yield decentralized asymptotically optimal strategies relative to the centralized strategy set, only applicable for the case of sufficiently large number of agents, the strategies generated by the de-aggregation method are exact decentralized optimal strategies relative to the decentralized strategy set and are applicable for arbitrary number of agents.

    $\bullet$ We embrace the de-aggregation approach to tackle the LQ Stackelberg MF games and teams problem, which differs from the fixed-point method used in articles like \cite{Moon-Basar-18} and \cite{Si-Wu-21}, etc., and the direct method employed in articles such as \cite{Wang-24} and \cite{Cong-Shi-24a}, etc.. The decentralized Stackelberg-Nash or Stackelberg-team equilibrium strategies are constructed.

    The paper is structured as follows.
    In Section 2, we formulate the problem of LQ Stackelberg MF games and teams.
    In Section 3, we design decentralized Stackelberg-Nash equilibrium in two parts.
    Initially, we take the leader's strategy as given and tackle the game problem for the followers.
    Subsequently, we deal with the leader's optimal control issue.
    In Section 4, following a similar approach to Section 3, we design decentralized Stackelberg-team equilibrium.
    Section 5 engages in a numerical simulation.
    Finally, some conclusions and future research directions are given in Section 6.

    The following notations will be used throughout this paper.
    We use $||\cdot||$ to denote the norm of a Euclidean space, or the Frobenius norm for matrices.
    The superscript $\top$ denotes the transpose of a vector or matrix.
    For a symmetric matrix $Q$ and a vector $z$, $||z||^2_Q \equiv z^\top Qz$, and $Q > 0$ $(Q \geq 0)$ means that $Q$ is positive definite (positive semi-definite).
    For two vectors $x,y$, $\langle x,y \rangle = x^\top y$.
    Let $T>0$ be a finite time duration, $C([0,T],\mathbb{R}^n)$ is the space of all $\mathbb{R}^n$-value continuous functions defined on $[0,T]$.
    Let $(\Omega, \mathcal{F}, \{\mathcal{H}_t\}_{0 \le t \le T}, \mathbb{P})$ be a complete filtered probability space with the filtration $\{\mathcal{H}_t\}_{0 \le t \le T}$ augmented by all the $\mathbb{P}$-null sets in $\mathcal{F}$.
    $\mathbb{E}[\cdot]$ denoted the expectation with respect to $\mathbb{P}$.
    Let $\mathcal{G}_t$ be some sub-$\sigma$-algebra of $\mathcal{H}_t$, $\mathbb{E}[\cdot|\mathcal{G}_t]$ denoted the conditional expectation with respect to filter $\mathcal{G}_t$.
    Let $L_{\mathcal{H}}^2(0,T;\cdot)$ be the set of all vector-valued (or matrix-valued) $\mathcal{H}_t$-adapted processes $f(\cdot)$ such that $\mathbb{E}\big[\int^T_0 ||f(t)||^2dt\big] < \infty$ and $L^2_{\mathcal{H}_t}(\Omega;\cdot)$ be the set of $\mathcal{H}_t$-measurable random variables, for $t\in[0,T]$.

    \section{Problem formulation}

    This paper presents a multi-agent system, elucidating the roles of one leader and $N$ exchangeable followers within its structure.
    Players are termed ``exchangeable" when their dynamics and cost functions do not change regardless of the indexing applied.
    Unlike previous studies, here $N$ can be arbitrarily large or a finite number.
    The states equation of the leader and the $i$th follower, $1 \le i \le N$, are given by the following controlled linear SDEs, respectively:
    \begin{equation}\label{state}
        \left\{
        \begin{aligned}
            dx_0(t) &= \big[A_0x_0(t) + B_0u_0(t) + f_0(t)\big]dt + D_0dW_0(t),\\
            dx_i(t) &= \big[Ax_i(t) + Bu_i(t) + f(t)\big]dt + DdW_i(t),\quad t\in[0,T],\\
            x_0(0) &= \xi_0 ,\quad x_i(0) = \xi_i , \quad i = 1,\cdots,N,
        \end{aligned}
        \right.
    \end{equation}
    where $x_0(\cdot) \in \mathbb{R}^n$, $u_0(\cdot)\in \mathbb{R}^m$ are the state process and the control process, $\xi_0$ is the initial value of the leader; similarly, $x_i(\cdot) \in \mathbb{R}^n$, $u_i(\cdot) \in \mathbb{R}^m$ and $\xi_i$ are the state process, control process and initial value of the $i$th follower.
    Here, $A_0, B_0, D_0, A, B, D$ are constant matrices with appropriate dimension, the non-homogeneous term $f_0, f \in C([0,T], \mathbb{R}^n)$.
    $W_i(\cdot), i = 0,\cdots,N$ are a sequence of one-dimensional Brownian motions defined on $(\Omega, \mathcal{F}, \{\mathcal{F}_t\}_{0 \le t \le T}, \mathbb{P})$.
    Let $\mathcal{F}_t$ be the $\sigma$-algebra generated by $\{\xi_i, W_i(s), s \le t, 0 \le i \le N\}$.
    Denote $\mathcal{F}_t^i$ be the sub-$\sigma$-algebra generated by $\{\xi_i, W_i(s), s \le t\}, 0 \le i \le N$.
    Define the decentralized control set for each agent as
    $$\mathscr{U}_{d,i}[0,T] := \big\{u_i(\cdot)|u_i(\cdot) \in L^2_{\mathcal{F}^i}(0,T;\mathbb{R}^m)\big\},\quad i = 0,\cdots,N.$$

    The cost functional of the leader and $i$th follower are given by
    \begin{equation}\label{cost of leader}
    \hspace{-0.3mm}    J^N_0(u_0(\cdot); u^N(\cdot)) = \frac{1}{2} \mathbb{E} \int_0^T \left[||x_0(t) - \Gamma_0x^{(N)}(t) - \eta_0(t)||^2_{Q_0} + ||u_0(t)||^2_{R_0}\right]dt,
    \end{equation}
    \begin{equation}\label{cost of followers}
        J^N_i(u_i(\cdot); u_{-i}(\cdot), u_0(\cdot))
        = \frac{1}{2} \mathbb{E} \int_0^T \left[||x_i(t) - \Gamma x^{(N)}(t) - \Gamma_1 x_0(t) - \eta(t)||^2_Q + ||u_i(t)||^2_R\right]dt,
        \end{equation}
    where $u^N(\cdot) := (u_1(\cdot),\cdots,u_N(\cdot))$, $u_{-i}(\cdot) := (u_1(\cdot),\cdots,u_{i-1}(\cdot),u_{i+1}(\cdot),\cdots,u_N(\cdot))$, $x^{(N)}(\cdot):= \frac{1}{N} \sum_{i=1}^{N} x_i(\cdot)$ is called the state average or MF term of all followers, $Q_0$, $R_0$, $Q$ and $R$ are symmetric constant matrices with appropriate dimension, $\Gamma_0$, $\Gamma$ and $\Gamma_1$ are constant matrices, and the non-homogeneous term $\eta_0, \eta \in C([0,T], \mathbb{R}^n)$.
    Notice that $Q_0$, $\Gamma_0$ and $\eta_0$ determine the coupling between the leader and the MF of the $N$ followers, $R_0$ is the control performance weighting parameter of the leader.
    $Q$, $\Gamma$, $\Gamma_1$ and $\eta$ determine the coupling between the $i$th follower, leader and followers' MF.
    Also, $R$ serves as the control performance weighting parameter of the $i$th follower.
    It is noteworthy that the followers are (weakly) coupled with each other through the MF term $x^{(N)}$, and are (strongly) coupled with the leader's state $x_0$ included in their cost functionals.

    \begin{Remark}
        The MF term $x^{(N)}(\cdot)$ here is widely used in classical MF games, and a more general weighted form $x^{(\alpha)}(\cdot) = \sum_{i=1}^{N} \alpha_i^{(N)} x_i(\cdot)$, where $\alpha_i^{(N)} \ge 0, i=1,\cdots,N$, $\sum_{i=1}^{N} \alpha_i^{(N)}=1$, can also be considered. For simplicity, we only discuss $x^{(N)}(\cdot)$ in this context.
    \end{Remark}

    We posit the following assumptions:

    \noindent {\bf (A1)} $\{\xi_i\}, i = 1,2,\cdots,N$ are a sequence of i.i.d. random variables, $\xi_0$ is a random variable independent of $\{\xi_i\}, i = 1,2,\cdots,N$,
    with $\mathbb{E}[\xi_i] = \bar{\xi}, i = 1,2,\cdots,N$, $\mathbb{E}[\xi_0] = \bar{\xi_0}$, and there exists a constant $c$ such that $\sup_{0 \le i \le N} \mathbb{E} [||\xi_i||^2] \le c.$

    \noindent {\bf (A2)} $\{W_i(t), 0 \le i \le N \}$ are independent of each other, which are
    also independent of $\{\xi_i, 0 \le i \le N \}$.

    \noindent {\bf (A3)} $Q \ge 0, R > 0$.

    \noindent {\bf (A4)} $Q_0 \ge 0, R_0 > 0$.

    Now, we present the precise definition of a decentralized Stackelberg-Nash equilibrium and a decentralized Stackelberg-team equilibrium.

    \begin{mydef}\label{decentralized Stackelberg-Nash equilibrium}
        A set of strategies $(u_0^*(\cdot),u_1^*(\cdot),\cdots,u_N^*(\cdot))$ is a decentralized Stackelberg-Nash equilibrium with respect to $\{J^N_i,0 \le i \le N\}$ if the following hold:

        (i) For a given strategy of the leader $u_0(\cdot) \in \mathscr{U}_{d,0}[0,T]$, $u^{N*}(\cdot) = (u_1^*(\cdot),\cdots,u_N^*(\cdot))$ constitutes a Nash equilibrium, if for all $i, 1 \le i \le N$,
        $$J^N_i(u_i^*(\cdot); u_{-i}^*(\cdot), u_0(\cdot)) = \inf_{u_i(\cdot) \in \mathscr{U}_{d,i}[0,T]} J^N_i(u_i(\cdot); u_{-i}^*(\cdot), u_0(\cdot));$$

        (ii)
        $$J^N_0(u_0^*(\cdot); u^{N*}[\cdot;u_0^*(\cdot)]) = \inf_{u_0(\cdot) \in \mathscr{U}_{d,0}[0,T]} J^N_0(u_0(\cdot); u^{N*}[\cdot;u_0(\cdot)]).$$
    \end{mydef}

    \begin{mydef}\label{decentralized Stackelberg-team equilibrium}
        A set of strategies $(u_0^*(\cdot),u_1^*(\cdot),\cdots,u_N^*(\cdot))$ is a decentralized Stackelberg-team equilibrium with respect to $\{J^N_i,0 \le i \le N\}$ if the following hold:

        (i) For a given strategy of the leader $u_0(\cdot) \in \mathscr{U}_{d,0}[0,T]$, $u^{N*}(\cdot) = (u_1^*(\cdot),\cdots,u_N^*(\cdot))$ constitutes a social optimal solution, if
        $$J^N_{soc}(u^{N*}(\cdot); u_0(\cdot)) = \inf_{u_i(\cdot) \in \mathscr{U}_{d,i}[0,T], 1 \le i \le N} J^N_{soc}(u^N(\cdot); u_0(\cdot)),$$
        where $J^N_{soc}(u^N) := \frac{1}{N} \sum_{i=1}^{N} J^N_i(u_i(\cdot); u_{-i}(\cdot), u_0(\cdot));$

        (ii)
        $$J^N_0(u_0^*(\cdot); u^{N*}[\cdot;u_0^*(\cdot)]) = \inf_{u_0(\cdot) \in \mathscr{U}_{d,0}[0,T]} J^N_0(u_0(\cdot); u^{N*}[\cdot;u_0(\cdot)]).$$
    \end{mydef}

    In this paper, we investigate the following problems.

    \noindent {\bf (PG)}: Find a decentralized Stackelberg-Nash equilibrium solution $(u^*_i(\cdot) \in \mathscr{U}_{d,i}[0,T], i = 0,\cdots,N)$ to (\ref{cost of leader}), (\ref{cost of followers}), subject to (\ref{state}).

    \noindent {\bf (PS)}: Find a decentralized Stackelberg-team equilibrium solution $(u^*_i(\cdot) \in \mathscr{U}_{d,i}[0,T], i = 0,\cdots,N)$ to (\ref{cost of leader}), (\ref{cost of followers}), subject to (\ref{state}).

    We need to point out that the methods for handling the two problems are highly similar. We will mainly discuss Problem {\bf (PG)}, with corresponding conclusions provided for Problem {\bf (PS)}.

    Based on the model in \cite{Carmona-Dayanikli-Lauriere-22}, we introduce a carbon emissions problem as a case study to elucidate the motivational underpinnings and practical context of Problem {\bf (PG)} and {\bf (PS)}.
    \begin{example}
        Considering the carbon emissions model, which comprises $N$ electricity productions, and a regulator.

        For symmetry reasons, we assume that the total electricity demand is split equally between all the electricity production agents, and each agent faces the same demand, say $D$.
        The dynamical system for producer $i$ is delineated by the following equations:
        \begin{equation*}
            \left\{
            \begin{aligned}
                dQ_i(t) =& c_1 \mathcal{N}_i(t)dt + c_2dS_i(t),\\
                dS_i(t) =& (\theta - S_i(t))dt + \sigma dW_i(t),\\
                dE_i(t) =& \delta \mathcal{N}_i(t)dt,\\
                d\tilde{\mathcal{N}}_i(t) =& \mathcal{N}_i(t)dt.
            \end{aligned}
            \right.
        \end{equation*}
        The first equation simply states that the instantaneous electricity production $Q_i(t)$ changes depend on the change rate of instantaneous nonrenewable energy usage $\mathcal{N}_i(t)$ and the change rate of instantaneous yield from the renewable energy investment $dS_i(t)$, where $c_1,c_2>0$ are constants that give the efficiency of the production from nonrenewable and renewable energy, respectively.
        We assume that $S_i(t), i=1,\cdots,N$ are Ornstein-Uhlenbeck processes with the same mean $\theta>0$ and volatility $\sigma>0$, the $W_i(t)$ being independent Wiener processes.
        The dynamics of the instantaneous emissions $E_i(t)$ are determined by the third equation.
        The choice of the constant $\delta$ could include the effects of some abatement measures such as carbon capture, sequestration and the use of filters.
        The instantaneous nonrenewable energy production $\tilde{\mathcal{N}}_i(t)$ is given by the fourth equation.
        Producer $i$ controls their state by choosing the rate of change $\mathcal{N}_i(t)$ in nonrenewable energy production.
        The expected cost to producer $i$ over the whole period is:
        \begin{equation*}
            \begin{aligned}
            J^N_i(\mathcal{N}_i(\cdot), \mathcal{N}_{-i}(\cdot)) = \frac{1}{2} \mathbb{E} \int_0^T &\left[ \alpha_1 \Big| Q_i(t)-D \Big|^2 - \alpha_2 \left( \rho_0 + \rho_1 \left( D - Q^{(N)}(t) \right)\right) Q_i(t) \right.\\
            & \left.+ \tau(t) E_i(t) + \alpha_3 \Big| \mathcal{N}_i(t)\Big|^2 + \alpha_4 \tilde{\mathcal{N}}_i(t) \right]dt,
            \end{aligned}
        \end{equation*}
        where $Q^{(N)}(\cdot):= \frac{1}{N} \sum_{i=1}^{N} Q_i(\cdot)$.
        Term 1 with $\alpha_1 > 0$ imposes a penalty on producers for not matching the demand and forcing the system operator to use costly reserves.
        Term 2 represents the revenues from electricity production, $\left( \rho_0 + \rho_1 \left( D - Q^{(N)}(t) \right)\right)$ being the inverse demand function which is assumed to be linear in excess demand or supply, where $\rho_0$ and $\rho_1$ are strictly positive constants.
        It captures the fact that the price increases if there is excess demand, and it decreases if there is excess supply.
        This term introduces the MF interactions into the model.
        Term 3 gives the pollution damage cost for the producer.
        This cost is levied by the regulator by using a carbon tax.
        Term 4 with $\alpha_3 > 0$ is a penalty (i.e., delay cost) for attempting to ramp up and down nonrenewable energy power plants too quickly.
        Term 5 represents the costs of the fossil fuels used in nonrenewable power plants.
        The constant $\alpha_4 > 0$ can be understood as the average cost of one unit of fossil fuel.

        The dynamical system for the regulator is delineated by the following equations:
        \begin{equation*}
            d\tau(t) = u(t) dt + \sigma_0 dW_0(t).
        \end{equation*}
        Regulator adjust the amount of carbon tax by controlling $u(t)$ to minimize their own cost:
        \begin{equation*}
            J^N_0(u(\cdot), \mathcal{N}^{(N)}(\cdot)) = \frac{1}{2} \mathbb{E} \int_0^T \left[ \beta_1 \Big| NE^{(N)}(t) \Big|^2 - \beta_2 \tau(t) E_i(t) + \beta_3 \Big| D - Q^{(N)}(t) \Big|^2  + \beta_4 \Big| u(t)\Big|^2 \right]dt,
        \end{equation*}
        where $\beta_1,\beta_2,\beta_3$ and $\beta_4$ are nonnegative constants whose role is explained below.
        The first term is the cost of total pollution.
        The second term is the revenue from the carbon tax.
        The roles of Term 3 is to ensure that the responsibility of matching the demand is not only incumbent on the producers, but also on the regulator, influencing the choice of $\beta_3$.
        This is consistent with the characterization of producers/regulator as a policy maker as well as a system operator bearing the brunt of managing the ancillary services to avoid disruptions like system blackouts.
        The fourth term represents a penalty for frequent policy fluctuations.

        In accordance with the carbon tax policy announced by the regulator, producers engage in games or cooperation to minimize their individual or social cost.
        Subsequently, the regulator selects a tax level that minimizes its own cost.
        Evidently, this problem can be construed as a LQ Stackelberg MF games and teams.

    \end{example}

    \section{LQ Stackelberg MF games}

    \subsection{MF games for the $N$ followers}

    In this subsection, we consider the multiagent Nash game for the $N$ followers under an arbitrary strategy of the leader $u_0(\cdot) \in \mathscr{U}_{d,0}[0,T]$.
    We suppose $u_0(\cdot)$ is fixed. Then, due to first equation of ({\ref{state}}), $x_0(\cdot)$ is also fixed.
    We now consider the following game problem for $N$ followers.

    \noindent {\bf (PG1)}: Minimize $J^N_i,i = 1,\cdots,N$ of (\ref{cost of followers}) over $u_i(\cdot) \in \mathscr{U}_{d,i}[0,T]$.

    For the sake of simplicity, the time variable $t$ will be omitted without ambiguity, and denote $\mathbb{E}_i[\cdot] := \mathbb{E}[\cdot|\mathcal{F}_t^i]$.

    We have the following result.

    \begin{mythm}\label{followers thm1}
        Under Assumption (A1)-(A2), let $u_0(\cdot) \in \mathscr{U}_{d,0}[0,T]$ be given, for the initial value $\xi_i, i = 1,\cdots,N$, {\bf (PG1)} admits an optimal control $\check{u}_i(\cdot) \in \mathscr{U}_{d,i}[0,T], i = 1,\cdots,N$, if and only if the following two conditions hold:

        (i) The adapted solution $\left( \check{x}_i(\cdot), \check{p}_i(\cdot), \check{q}^j_i(\cdot), i = 1,\cdots,N, j = 0,1,\cdots,N \right)$ to the FBSDE
        \begin{equation}\label{adjoint FBSDE of followers}
            \left\{
            \begin{aligned}
                d\check{x}_i &= \big[A\check{x}_i + B\check{u}_i + f\big]dt + DdW_i,\\
                d\check{p}_i &= -\left[A^\top \check{p}_i + \left(I - \frac{\Gamma}{N}\right)^\top Q\left(\check{x}_i - \Gamma \check{x}^{(N)} - \Gamma_1 x_0 - \eta\right)\right]dt + \sum_{j=0}^{N} \check{q}_i^j dW_j,\quad t\in[0,T],\\
                \check{x}_i(0) &= \xi_i ,\quad \check{p}_i(T) = 0 , \quad i = 1,\cdots,N,
            \end{aligned}
            \right.
        \end{equation}
        satisfies the following stationarity condition:
        \begin{equation}\label{stationarity condition of followers}
            B^\top \mathbb{E}_i[\check{p}_i] + R\check{u}_i = 0, \quad i = 1,\cdots,N.
        \end{equation}

        (ii) For $i = 1,\cdots,N$, the following convexity condition holds:
        \begin{equation}\label{convexity condition of followers}
            \mathbb{E} \int_0^T \left[ ||\left(I - \frac{\Gamma}{N}\right) \tilde{x}_i||^2_Q + ||u_i||^2_R \right]dt \ge 0,\ i = 1,\cdots,N, \quad \forall u_i(\cdot) \in \mathscr{U}_{d,i}[0,T],
        \end{equation}
        where $\tilde{x}_i(\cdot)$ is the solution to the following {\it random differential equation} (RDE):
        \begin{equation}\label{random differential equation of followers}
            \left\{
            \begin{aligned}
                &d\tilde{x}_i = [A\tilde{x}_i + Bu_i]dt,\\
                &\tilde{x}_i(0) = 0.
            \end{aligned}
            \right.
        \end{equation}
    \end{mythm}

    \begin{proof}
        We consider the $i$th follower.
        For given $\xi_i$, $u_0(\cdot) \in \mathscr{U}_{d,0}[0,T]$, $\check{u}_i(\cdot) \in \mathscr{U}_{d,i}[0,T]$, suppose that $(\check{x}_i(\cdot), \check{p}_i(\cdot), \check{q}^j_i(\cdot), j = 0,1,\cdots,N)$ is an adapted solution to FBSDE (\ref{adjoint FBSDE of followers}).
        For any $u_i(\cdot) \in \mathscr{U}_{d,i}[0,T]$ and $\varepsilon \in \mathbb{R}$, let $x_i^{\varepsilon}(\cdot)$ be the solution to the following perturbed state equation:
        \begin{equation*}
            \left\{
            \begin{aligned}
                dx_i^{\varepsilon} &= \big[Ax_i^{\varepsilon} + B(\check{u}_i + \varepsilon u_i) + f\big]dt + DdW_i,\\
                x_i^{\varepsilon}(0) &= \xi_i.
            \end{aligned}
            \right.
        \end{equation*}
        Then, $\tilde{x}_i(\cdot) := \frac{x_i^{\varepsilon}(\cdot) - \check{x}_i(\cdot)}{\varepsilon}$ is independent of $\varepsilon$ and satisfies (\ref{random differential equation of followers}).

        Applying It\^{o}'s formula to $\langle \check{p}_i(\cdot), \tilde{x}_i(\cdot) \rangle$, integrating from 0 to $T$, and taking the expectation, we have
        \begin{equation*}
        \begin{aligned}
            0 &= \mathbb{E} \big[\langle \check{p}_i(T), \tilde{x}_i(T) \rangle - \langle \check{p}_i(0), \tilde{x}_i(0) \rangle \big] \\
            &= \mathbb{E} \int_0^T \left[ -\left\langle A^\top \check{p}_i + \left(I - \frac{\Gamma}{N}\right)^\top Q\left(\check{x}_i - \Gamma \check{x}^{(N)}
              - \Gamma_1 x_0 - \eta\right), \tilde{x}_i \right\rangle + \left\langle \check{p}_i, A\tilde{x}_i + Bu_i \right\rangle \right]dt \\
            &= \mathbb{E} \int_0^T \left[ -\left\langle \left(I - \frac{\Gamma}{N}\right)^\top Q\left(\check{x}_i - \Gamma \check{x}^{(N)} - \Gamma_1 x_0 - \eta\right), \tilde{x}_i \right\rangle
             + \left\langle \check{p}_i, Bu_i \right\rangle \right]dt.
        \end{aligned}
        \end{equation*}
        Hence,
        \begin{equation*}
        \begin{aligned}
            &J^N_i(\check{u}_i(\cdot) + \varepsilon u_i(\cdot); u_{-i}(\cdot), u_0(\cdot)) - J^N_i(\check{u}_i(\cdot); u_{-i}(\cdot), u_0(\cdot))\\
            &= \frac{1}{2} \mathbb{E} \int_0^T \left[||x_i^{\varepsilon} - \Gamma x^{\varepsilon(N)} - \Gamma_1 x_0 - \eta||^2_Q + ||\check{u}_i + \varepsilon u_i||^2_R \right]dt \\
            &\quad - \frac{1}{2} \mathbb{E} \int_0^T \left[||\check{x}_i - \Gamma \check{x}^{(N)} - \Gamma_1 x_0 - \eta||^2_Q + ||\check{u}_i||^2_R \right]dt \\
            &= \frac{1}{2} \mathbb{E} \int_0^T \left[||\left(I - \frac{\Gamma}{N}\right)(\check{x}_i + \varepsilon \tilde{x}_i) - \Gamma x^{(N-1)}_{-i} - \Gamma_1 x_0 - \eta||^2_Q
             + ||\check{u}_i + \varepsilon u_i||^2_R \right]dt \\
            &\quad - \frac{1}{2} \mathbb{E} \int_0^T \left[||\left(I - \frac{\Gamma}{N}\right)\check{x}_i - \Gamma x^{(N-1)}_{-i} - \Gamma_1 x_0 - \eta||^2_Q + ||\check{u}_i||^2_R \right]dt \\
            &= \frac{1}{2} \varepsilon^2 \mathbb{E} \int_0^T \bigg[||\left(I - \frac{\Gamma}{N}\right) \tilde{x}_i||^2_Q + ||u_i||^2_R \bigg]dt \notag \\
            &\quad + \varepsilon \mathbb{E} \int_0^T \left[ \left\langle \left(I - \frac{\Gamma}{N}\right)^\top Q  \left(\check{x}_i - \Gamma \check{x}^{(N)}
             - \Gamma_1 x_0 - \eta\right), \tilde{x}_i \right\rangle + \left\langle R\check{u}_i, u_i\right\rangle \right]dt \\
            &= \frac{1}{2} \varepsilon^2 \mathbb{E} \int_0^T \bigg[||\left(I - \frac{\Gamma}{N}\right) \tilde{x}_i||^2_Q + ||u_i||^2_R \bigg]dt
             + \varepsilon \mathbb{E} \int_0^T \left\langle B^\top \check{p}_i + R\check{u}_i, u_i\right\rangle dt,
        \end{aligned}
        \end{equation*}
        where only here $x^{\varepsilon(N)}(\cdot) := \frac{x_i^{\varepsilon}(\cdot)}{N} + \frac{1}{N} \sum_{j \ne i} x_j(\cdot)$, $\check{x}^{(N)}(\cdot) := \frac{\check{x}_i(\cdot)}{N} + \frac{1}{N} \sum_{j \ne i} x_j(\cdot)$, and $x^{(N-1)}_{-i}(\cdot) = \frac{1}{N} \sum_{j \ne i} x_j(\cdot)$.
        Therefore,
        $$J^N_i(\check{u}_i(\cdot); u_{-i}(\cdot), u_0(\cdot)) \le J^N_i(\check{u}_i(\cdot) + \varepsilon u_i(\cdot); u_{-i}(\cdot), u_0(\cdot))$$
        if and only if (\ref{stationarity condition of followers}) and (\ref{convexity condition of followers}) hold. The proof is complete.
    \end{proof}

    \begin{Remark}
        Assumption (A3) is a special case of the convexity condition presented in equation (\ref{convexity condition of followers}).
        Due to its ease of verification, this form of the convexity condition, as encapsulated in Assumption (A3), is widely adopted in the literature.
    \end{Remark}

    Based on (A3), we can figure out that the open-loop optimal strategies for followers are
    \begin{equation}\label{open-loop optimal strategies of followers}
        \check{u}_i = -R^{-1} B^\top \mathbb{E}_i[\check{p}_i],\quad i = 1,\cdots,N.
    \end{equation}
    Thus, the pertinent Hamiltonian system can be formulated as
    \begin{equation}\label{Hamiltonian system of followers}
        \left\{
        \begin{aligned}
            d\check{x}_i &= \big[A\check{x}_i - BR^{-1}B^\top \mathbb{E}_i[\check{p}_i] + f\big]dt + DdW_i,\\
            d\check{p}_i &= -\left[A^\top \check{p}_i + \left(I - \frac{\Gamma}{N}\right)^\top Q \left(\check{x}_i - \Gamma \check{x}^{(N)} - \Gamma_1 x_0 - \eta\right)\right]dt + \sum_{j=0}^{N} \check{q}_i^j dW_j,\\
            \check{x}_i(0) &= \xi_i ,\quad \check{p}_i(T) = H\check{x}_i(T) , \quad i = 1,\cdots,N.
        \end{aligned}
        \right.
    \end{equation}

    Due to the appearance of the $\mathbb{E}_i[\check{p}_i]$ term in (\ref{open-loop optimal strategies of followers}), we need to find the equation of $\mathbb{E}_i[\check{p}_i]$. Before that, we will introduce the following content.

    According to (\ref{state}), it can be inferred that $\check{x}_i$ is adapted to $\mathcal{F}_t^i, i = 0,\cdots,N.$
    And $N$ followers are exchangeable, by taking conditional expectations $\mathbb{E}_i[\cdot], i = 1, \cdots,N$, we have
    \begin{equation}\label{conditional expectation for Fi}
        \begin{aligned}
            &\mathbb{E}_i[\check{x}_j] = \mathbb{E}[\check{x}_j] = \mathbb{E}[\check{x}_i], \quad 1 \le j \ne i \le N;\\
            &\mathbb{E}_i[x_0] = \mathbb{E}[x_0]; \quad \mathbb{E}_i[\check{x}_i] = \check{x}_i.
    \end{aligned}
    \end{equation}

    Then, the MF term $\check{x}^{(N)}$ can be de-aggregated into linear form of $\check{x}_i$ and the expectation of $\check{x}_i$ as
    \begin{equation}\label{xN conditional expectation for Fi}
        \mathbb{E}_i[\check{x}^{(N)}] = \frac{1}{N} \check{x}_i + \frac{1}{N} \sum_{1 \le j \ne i \le N} \mathbb{E}[\check{x}_j] = \frac{1}{N} \check{x}_i + \frac{N-1}{N} \mathbb{E}[\check{x}_i],
    \end{equation}
    where
    \begin{equation}\label{xi expectation}
        d\mathbb{E}[\check{x}_i] = [A\mathbb{E}[\check{x}_i] - R^{-1} B^\top \mathbb{E}[\check{p}_i] + f]dt,\quad \mathbb{E}[\check{x}_i(0)] = \bar{\xi}.
    \end{equation}

    \begin{Remark}
        This part draws upon the methodologies presented in references \cite{Wang-Zhang-Zhang-22}, \cite{Wang-Zhang-Fu-Liang-23} and \cite{Liang-Wang-Zhang-24}.
        By taking the conditional expectation over the filtration adapted to the decentralized control set for the $i$-th agent, we have decomposed the MF aggregation into a linear form that encompasses both the state of the $i$-th agent and its expectation.
        This decomposition decouples the inter-agent coupling, independent of the number of agents $N$, thereby ensuring that the optimal strategy for any number of agents $N$ is an exact Nash equilibrium.
        We encapsulate this approach as the de-aggregation method.
    \end{Remark}

    Applying Lemma 5.4 in \cite{Xiong-08}, for $i = 1,\cdots,N$ we have
    \begin{equation}\label{qij conditional expectation for Fi}
        \begin{aligned}
            &\mathbb{E}_i \left[ \int_0^t \check{q}_i^i dW_i \right] = \int_0^t \mathbb{E}_i[\check{q}_i^i] dW_i,\\
            &\mathbb{E}_i \left[ \int_0^t \check{q}_i^j dW_j \right] = 0, \quad 0 \le j \ne i \le N.
        \end{aligned}
    \end{equation}
    We can provide the equation of $\mathbb{E}_i[\check{p}_i]$ below:
    \begin{equation}\label{pi conditional expectation for Fi}
    \left\{
        \begin{aligned}
            d\mathbb{E}_i[\check{p}_i] &= -\bigg[A^\top \mathbb{E}_i[\check{p}_i] + \left(I - \frac{\Gamma}{N}\right)^\top Q \bigg( \left(I - \frac{\Gamma}{N}\right) \check{x}_i - \frac{N-1}{N} \Gamma \mathbb{E}[\check{x}_i]\\
            &\qquad - \Gamma_1 \mathbb{E}[x_0] - \eta\bigg)\bigg]dt + \mathbb{E}_i[\check{q}_i^i] dW_i, \\
            \mathbb{E}_i[\check{p}_i(T)] &= 0.
    \end{aligned}
    \right.
    \end{equation}

    Consider the transformation
    \begin{equation}\label{transformation of adjoint equation}
        \mathbb{E}_i[\check{p}_i(\cdot)] = P_N(\cdot)\check{x}_i(\cdot) + K_N(\cdot)\mathbb{E}[\check{x}_i](\cdot) + \check{\phi}_N(\cdot), \quad i = 1,\cdots,N,
    \end{equation}
    where $P_N(\cdot)$, $K_N(\cdot)$ and $\check{\phi}_N(\cdot)$ are differential functions with $P_N(T) = 0$, $K_N(T) = 0$ and $\check{\phi}_N(T) = 0$.
    By It\^{o}'s formula, we get
    \begin{equation*}
    \begin{aligned}
        d\mathbb{E}_i[\check{p}_i] = &\left(\dot{P}_N \check{x}_i + \dot{K}_N \mathbb{E}[\check{x}_i] + \dot{\check{\phi}}_N \right) dt \\
        &+ P_N \big[A\check{x}_i - BR^{-1}B^\top \mathbb{E}_i[\check{p}_i] + f\big]dt + P_N DdW_i \\
        &+ K_N \left[A\mathbb{E}[\check{x}_i] - R^{-1} B^\top \mathbb{E}[\check{p}_i] + f \right]dt.
    \end{aligned}
    \end{equation*}
    Comparing the coefficients of the corresponding terms with (\ref{pi conditional expectation for
    Fi}), we have
    $$\mathbb{E}_i[\check{q}_i^i] = P_ND,$$
    \begin{equation}\label{Riccati-1 of followers}
        \dot{P}_N + A^\top P_N + P_NA - P_NBR^{-1}B^\top P_N + \left(I - \frac{\Gamma}{N}\right)^\top Q \left(I - \frac{\Gamma}{N}\right) = 0,\quad P_N(T) = 0,
    \end{equation}
    \begin{equation}\label{Riccati-2 of followers}
        \begin{aligned}
            &\dot{K}_N + A^\top K_N + K_NA - P_NBR^{-1}B^\top K_N - K_NBR^{-1}B^\top (P_N + K_N) \\
            &\qquad- \left(I - \frac{\Gamma}{N}\right)^\top Q \frac{N-1}{N}\Gamma = 0,\quad K_N(T) = 0,
        \end{aligned}
    \end{equation}
    \begin{equation}\label{BODE-1 of followers}
    \hspace{-3cm}    \begin{aligned}
            \dot{\check{\phi}}_N + \left[A^\top - (P_N + K_N)BR^{-1}B^\top \right] \check{\phi}_N + (P_N + K_N)f \\
            - \left(I - \frac{\Gamma}{N}\right)^\top Q (\Gamma_1 \mathbb{E}[x_0] + \eta) = 0,\quad \check{\phi}_N(T) = 0.
        \end{aligned}
    \end{equation}
    Let $\Pi_N(\cdot) := P_N(\cdot) + K_N(\cdot)$, then $\Pi_N(\cdot)$ satisfies
    \begin{equation}\label{Riccati-1+2 of followers}
        \dot{\Pi}_N + A^\top \Pi_N + \Pi_NA - \Pi_NBR^{-1}B^\top \Pi_N + \left(I - \frac{\Gamma}{N}\right)^\top Q \left( I - \Gamma \right) = 0,\quad \Pi_N(T) = 0,
    \end{equation}
    thus (\ref{BODE-1 of followers}) can be rewritten as
    \begin{equation}\label{BODE-2 of followers}
        \dot{\check{\phi}}_N + \left[A^\top - \Pi_N BR^{-1}B^\top \right] \check{\phi}_N + \Pi_N f - \left(I - \frac{\Gamma}{N}\right)^\top Q (\Gamma_1 \mathbb{E}[x_0] + \eta) = 0,\quad \check{\phi}_N(T) = 0.
    \end{equation}

    \begin{Remark}
        We can see that $P_N(\cdot)$, $K_N(\cdot)$ and $\check{\phi}_N(\cdot)$ are actually independent of $i$.
    \end{Remark}

    \begin{Remark}
        Note that (\ref{Riccati-1 of followers}) is a symmetric Riccati differential equation, if it satisfies (A3), then it admits a unique solution.
        Similarly, (\ref{Riccati-1+2 of followers}) is a symmetric Riccati differential equation, if it satisfies (A3) and $(1 - \frac{\Gamma}{N})^\top (1 - \Gamma) \ge 0$, then it admits a unique solution.
        Thus (\ref{Riccati-2 of followers}) admits a unique solution.
        And then (\ref{BODE-2 of followers}) admits a unique solution.
    \end{Remark}

    From the above discussion, Theorem \ref{followers thm1}, and Theorem 4.1 of \cite{Ma-Yong-99}, we have the following result.
    \begin{mythm}\label{followers thm2}
        Under Assumptions (A1)-(A3), and assume (\ref{Riccati-1 of followers}) and (\ref{Riccati-2 of followers}) admit a solution, respectively.
        Then, for given $u_0(\cdot) \in \mathscr{U}_{d,0}[0,T]$, Problem {\bf (PG1)} admits a unique solution
        \begin{equation}\label{optimal strategy of followers}
            \check{u}_i = -R^{-1} B^\top \left( P_N\check{x}_i + K_N\mathbb{E}[\check{x}_i] + \check{\phi}_N \right), \quad i = 1,\cdots,N,
        \end{equation}
        where
        $$d\check{x}_i = \left[(A - BR^{-1}B^\top P_N) \check{x}_i - BR^{-1}B^\top K_N \mathbb{E}[\check{x}_i] + f - BR^{-1}B^\top \check{\phi}_N \right] dt + DdW_i, \quad \check{x}_i(0) = \xi_i,$$
        $$d\mathbb{E}[\check{x}_i] = \left[(A - BR^{-1}B^\top \Pi_N) \mathbb{E}[\check{x}_i] + f - BR^{-1}B^\top \check{\phi}_N \right] dt, \quad \mathbb{E}[\check{x}_i(0)] = \bar{\xi}.$$
    \end{mythm}

    \begin{Remark}
        It is evident that the equation for $\mathbb{E}[\check{x}_i]$ here is consistent with the equation for the auxiliary function $\bar{x}$ obtained using other methods, like \cite{Cong-Shi-24a}.
    \end{Remark}

    \subsection{Optimal strategy of the leader}

    Upon implementing the strategies of followers $\check{u}_i(\cdot), i = 1,\cdots,N$ according to (\ref{optimal strategy of followers}), we delve into an optimal control problem for the leader.

    {\bf (PG2)}: Minimize $\check{J}^N_0(u_0(\cdot))$ over $u_0(\cdot) \in \mathscr{U}_{d,0}[0,T]$, where
    \begin{equation}\label{cost of leader-}
        \check{J}^N_0(u_0(\cdot)) := \frac{1}{2} \mathbb{E} \int_0^T \left[||x_0(t) - \Gamma_0\check{x}^{(N)}(t) - \eta_0||^2_{Q_0} + ||u_0(t)||^2_{R_0}\right]dt,
    \end{equation}
    subject to
    \begin{equation}\label{state of leader}
        \left\{
        \begin{aligned}
            dx_0 &= \left[A_0x_0 + B_0u_0 + f_0\right]dt + D_0dW_0,\\
            d\check{x}_i &= \left[(A - BR^{-1}B^\top P_N) \check{x}_i - BR^{-1}B^\top (\Pi_N - P_N) \mathbb{E}[\check{x}_i] + f - BR^{-1}B^\top \check{\phi}_N \right] dt + DdW_i,\\
            d\check{\phi}_N &= - \left[\left(A^\top - \Pi_N BR^{-1}B^\top \right) \check{\phi}_N + \Pi_N f - \left(I - \frac{\Gamma}{N}\right)^\top Q (\Gamma_1 \mathbb{E}[x_0] + \eta) \right] dt,\\
            x_0(0) &= \xi_0, \quad \check{x}_i(0) = \xi_i,\quad \check{\phi}_N(T) = 0, \quad i = 1,\cdots,N.
        \end{aligned}
        \right.
    \end{equation}
    Note that here $\check{x}^{(N)}(\cdot)$ means that all $x_i(\cdot), i = 1,\cdots,N$ take the optimal $\check{x}_i(\cdot)$, i.e., $\check{x}^{(N)}(\cdot) := \frac{1}{N} \sum_{i=1}^{N} \check{x}_i(\cdot)$.
    Denote $W^{(N)}(\cdot) := \frac{1}{N} \sum_{i=1}^{N} W_i(\cdot)$, and $\xi^{(N)} := \frac{1}{N} \sum_{i=1}^{N} \xi_i$.

    We present the subsequent outcome.

    \begin{mythm}\label{leader thm1}
        Under Assumptions (A1)-(A3), let the followers adopt the optimal strategy (\ref{optimal strategy of followers}). Then Problem {\bf (PG2)} admits an optimal control $\check{u}_0(\cdot) \in \mathscr{U}_{d,0}[0,T]$, if and only if the following two conditions hold:

        (i) The adapted solution $\left( \check{x}_0(\cdot), \check{x}^{(N)}(\cdot), \check{\phi}_N(\cdot), \check{y}_0(\cdot), \check{z}_0(\cdot), \check{z}(\cdot), \check{y}^{(N)}(\cdot), \check{z}_0^{(N)}(\cdot), \check{z}^{(N)}(\cdot), \check{\psi}_N(\cdot) \right)$ to the FBSDE
        \begin{equation}\label{adjoint FBSDE of leader}
            \left\{
            \begin{aligned}
                d\check{x}_0 = & \left[ A_0\check{x}_0 + B_0\check{u}_0 + f_0 \right]dt + D_0dW_0, \quad \check{x}_0(0) = \xi_0, \\
                d\check{x}^{(N)} = & \left[ (A - BR^{-1}B^\top P_N) \check{x}^{(N)} - BR^{-1}B^\top (\Pi_N - P_N) \mathbb{E}[\check{x}^{(N)}] + f - BR^{-1}B^\top \check{\phi}_N \right] dt \\
                &+ DdW^{(N)}, \quad \check{x}^{(N)}(0) = \xi^{(N)},\\
                d\check{\phi}_N = &\ -\left[ (A^\top - \Pi_NBR^{-1}B^\top)\check{\phi}_N + \Pi_Nf - \left(I - \frac{\Gamma}{N}\right)^\top Q (\Gamma_1 \mathbb{E}[\check{x}_0] + \eta) \right]dt, \\
                &\ \check{\phi}_N(T) = 0, \\
                d\check{y}_0 = &\ -\left[ A_0^\top \check{y}_0 + \Gamma_1^\top Q \left(I - \frac{\Gamma}{N}\right) \mathbb{E}[\check{\psi}_N] + Q_0(\check{x}_0 - \Gamma_0 \check{x}^{(N)} - \eta_0) \right]dt \\
                &\ + \check{z}_0 dW_0 + \check{z} dW^{(N)}, \quad \check{y}_0(T) = 0, \\
                d\check{y}^{(N)} = &\ -\big[ (A - BR^{-1}B^\top P_N)^\top \check{y}^{(N)} - (\Pi_N - P_N)BR^{-1}B^\top \mathbb{E}[\check{y}^{(N)}] \\
                &\ - \Gamma_0^\top Q_0(\check{x}_0 - \Gamma_0\check{x}^{(N)} - \eta_0) \big]dt + \check{z}_0^{(N)} dW_0 + \check{z}^{(N)} dW^{(N)}, \quad \check{y}^{(N)}(T) = 0, \\
                d\check{\psi}_N = &\ \left[ BR^{-1}B^\top \check{y}^{(N)} + (A - BR^{-1}B^\top \Pi_N)\check{\psi}_N \right]dt, \quad \check{\psi}_N(0) = 0,
            \end{aligned}
            \right.
        \end{equation}
        satisfies the following stationarity condition:
        \begin{equation}\label{stationarity condition of leader}
            B_0^\top \mathbb{E}_0[\check{y}_0] + R_0\check{u}_0 = 0.
        \end{equation}

        (ii) The following convexity condition holds:
        \begin{equation}\label{convexity condition of leader}
            \mathbb{E} \int_0^T \left[ ||\tilde{x}_0 - \Gamma_0\tilde{x}^{(N)}||_{Q_0}^2 + ||u_0||_{R_0}^2\right]dt \ge 0, \quad \forall u_0(\cdot) \in \mathscr{U}_{d,0}[0,T],
        \end{equation}
        where $\left(\tilde{x}_0(\cdot), \tilde{x}^{(N)}(\cdot), \tilde{\phi}_N(\cdot)\right)$ is the solution to the following FBSDE
        \begin{equation}\label{random FBSDE of leader}
            \left\{
            \begin{aligned}
                d\tilde{x}_0 = &\ \left[ A_0\tilde{x}_0 + B_0u_0 \right]dt, \quad \tilde{x}_0(0) = 0, \\
                d\tilde{x}^{(N)} = &\ \left[ (A - BR^{-1}B^\top P_N)\tilde{x}^{(N)} - BR^{-1}B^\top (\Pi_N - P_N) \mathbb{E}[\tilde{x}^{(N)}] - BR^{-1}B^\top \tilde{\phi}_N \right]dt,\\
                & \ \tilde{x}^{(N)}(0) = 0, \\
                d\tilde{\phi}_N = & -\left[(A^\top - \Pi_NBR^{-1}B^\top)\tilde{\phi}_N - \left(I - \frac{\Gamma}{N}\right)^\top Q \Gamma_1 \mathbb{E}[\tilde{x}_0] \right]dt, \quad \tilde{\phi}_N(T) = 0.
            \end{aligned}
            \right.
        \end{equation}
    \end{mythm}

    \begin{proof}
        For given $\xi_0$ and $\check{u}_0(\cdot) \in \mathscr{U}_{d,0}[0,T]$, let $\left(\check{x}_0(\cdot), \check{x}^{(N)}(\cdot), \check{\phi}_N(\cdot), \check{y}_0(\cdot), \check{z}_0(\cdot), \check{z}(\cdot), \check{y}^{(N)}(\cdot)\right.$,\\ $\left. \check{z}_0^{(N)}(\cdot), \check{z}^{(N)}(\cdot), \check{\psi}_N(\cdot)\right)$ be an adapted solution to FBSDE (\ref{adjoint FBSDE of leader}).
        For any $u_0(\cdot) \in \mathscr{U}_{d,0}[0,T]$ and $\varepsilon \in \mathbb{R}$, let $\left(x_0^{\varepsilon}(\cdot), x^{\varepsilon(N)}(\cdot), \phi_N^{\varepsilon}(\cdot)\right)$ be the solution to the following perturbed state equation of the leader:
        \begin{equation*}
            \left\{
            \begin{aligned}
                dx_0^{\varepsilon} = &\ \left[A_0x_0^{\varepsilon} + B_0(\check{u}_0 + \varepsilon u_0) + f_0\right]dt + D_0dW_0, \quad x_0^{\varepsilon}(0) = \xi_0, \\
                dx^{\varepsilon(N)} = &\ \left[(A - BR^{-1}B^\top P_N)x^{\varepsilon(N)} - BR^{-1}B^\top (\Pi_N - P_N) \mathbb{E}[x^{\varepsilon(N)}] + f - BR^{-1}B^\top \phi_N^{\varepsilon} \right]dt \\
                &\ + DdW^{(N)}, \quad x^{\varepsilon(N)}(0) = \xi^{(N)}, \\
                d\phi_N^{\varepsilon} = & -\left[(A^\top - \Pi_NBR^{-1}B^\top)\phi_N^{\varepsilon} + \Pi_N f - \left(I - \frac{\Gamma}{N}\right)^\top Q (\Gamma_1 \mathbb{E}[x_0^{\varepsilon}] + \eta)\right]dt, \quad \phi_N^{\varepsilon}(T) = 0. \\
            \end{aligned}
            \right.
        \end{equation*}
        Then, denoting by $\left(\tilde{x}_0(\cdot), \tilde{x}^{(N)}(\cdot), \tilde{\phi}_N(\cdot)\right)$ the solution to (\ref{random FBSDE of leader}), we have $x_0^{\varepsilon}(\cdot) = \check{x}_0(\cdot) + \varepsilon \tilde{x}_0(\cdot)$, $x^{\varepsilon(N)}(\cdot) = \check{x}^{(N)}(\cdot) + \varepsilon \tilde{x}^{(N)}(\cdot)$, $\phi_N^{\varepsilon}(\cdot) = \check{\phi}_N(\cdot) + \varepsilon \tilde{\phi}_N(\cdot)$, and
        \begin{equation*}
            \begin{aligned}
                &\check{J}^N_0(\check{u}_0(\cdot) + \varepsilon u_0(\cdot)) - \check{J}^N_0(\check{u}_0(\cdot))\\
                = &\ \frac{1}{2} \mathbb{E} \int_0^T \left[||x_0^{\varepsilon} - \Gamma_0\check{x}^{\varepsilon(N)} - \eta_0||^2_{Q_0} + ||\check{u}_0 + \varepsilon u_0||^2_{R_0}\right]dt \\
                &\ - \frac{1}{2} \mathbb{E} \int_0^T \left[||\check{x}_0 - \Gamma_0\check{x}^{(N)} - \eta_0||^2_{Q_0} + ||\check{u}_0||^2_{R_0}\right]dt \\
                = &\ \frac{1}{2} \mathbb{E} \int_0^T \left[||(\check{x}_0 + \varepsilon \tilde{x}_0) - \Gamma_0(\check{x}^{(N)} + \varepsilon \tilde{x}^{(N)}) - \eta_0||^2_{Q_0} + ||\check{u}_0 + \varepsilon u_0||^2_{R_0}\right]dt \\
                &\ - \frac{1}{2} \mathbb{E} \int_0^T \left[||\check{x}_0 - \Gamma_0\check{x}^{(N)} - \eta_0||^2_{Q_0} + ||\check{u}_0||^2_{R_0}\right]dt \\
                = &\ \frac{1}{2} \varepsilon^2 \mathbb{E} \int_0^T \left[ ||\tilde{x}_0 - \Gamma_0\tilde{x}^{(N)}||_{Q_0}^2 + ||u_0||_{R_0}^2\right]dt \\
                &\ + \varepsilon \mathbb{E} \int_0^T \left[ \left\langle Q_0(\check{x}_0 - \Gamma_0 \check{x}^{(N)} - \eta_0), \tilde{x}_0 - \Gamma_0\tilde{x}^{(N)} \right\rangle + \left\langle R_0\check{u}_0,u_0 \right\rangle \right]dt.
            \end{aligned}
        \end{equation*}
        On the other hand, applying It\^{o}'s formula to $\langle \check{y}_0(\cdot), \tilde{x}_0(\cdot) \rangle + \langle \check{y}^{(N)}(\cdot), \tilde{x}^{(N)}(\cdot) \rangle + \langle \check{\psi}_N(\cdot), \tilde{\phi}_N(\cdot) \rangle$, integrating from 0 to $T$ and taking expectation, we obtain
        \begin{equation*}
            \begin{aligned}
                0 = &\ \mathbb{E}[\langle \check{y}_0(T), \tilde{x}_0(T) \rangle - \langle \check{y}_0(0), \tilde{x}_0(0) \rangle + \langle \check{y}^{(N)}(T), \tilde{x}^{(N)}(T) \rangle
                 - \langle \check{y}^{(N)}(0), \tilde{x}^{(N)}(0) \rangle \\
                &\ + \langle \check{\psi}_N(T), \tilde{\phi}_N(T) \rangle] - \langle \check{\psi}_N(0), \tilde{\phi}_N(0) \rangle] \\
                = &\ \mathbb{E} \int_0^T \bigg[ \langle \check{y}_0, A_0\tilde{x}_0 + B_0u_0 \rangle - \left\langle A_0^\top \check{y}_0
                 + \Gamma_1^\top Q \left(I - \frac{\Gamma}{N}\right) \mathbb{E}[\check{\psi}_N] + Q_0(\check{x}_0 - \Gamma_0 \check{x}^{(N)} - \eta_0), \tilde{x}_0 \right\rangle \\
                &\ + \left\langle \check{y}^{(N)}, (A - BR^{-1}B^\top P_N)\tilde{x}^{(N)} - BR^{-1}B^\top (\Pi_N - P_N) \mathbb{E}[\tilde{x}^{(N)}] - BR^{-1}B^\top \tilde{\phi}_N \right\rangle \\
                &\ - \left\langle (A - BR^{-1}B^\top P_N)^\top \check{y}^{(N)} - (\Pi_N - P_N)BR^{-1}B^\top \mathbb{E}[\check{y}^{(N)}] \right.\\
                &\qquad \left.- \Gamma_0^\top Q_0(\check{x}_0 - \Gamma_0\check{x}^{(N)} - \eta_0), \tilde{x}^{(N)} \right\rangle
                 - \bigg\langle \check{\psi}_N, (A^\top - \Pi_NBR^{-1}B^\top)\tilde{\phi}_N \\
                &\qquad \left.- \left(I - \frac{\Gamma}{N}\right)^\top Q \Gamma_1 \mathbb{E}[\tilde{x}_0] \right\rangle
                 + \langle BR^{-1}B^\top \check{y}^{(N)} + (A - BR^{-1}B^\top \Pi_N)\check{\psi}_N, \tilde{\phi}_N \rangle \bigg]dt \\
                = &\ \mathbb{E} \int_0^T \left[ \left\langle B_0^\top\check{y}_0,u_0 \right\rangle - \left\langle Q_0(\check{x}_0 - \Gamma_0 \check{x}^{(N)} - \eta_0), \tilde{x}_0 - \Gamma_0\tilde{x}^{(N)} \right\rangle  \right]dt.
            \end{aligned}
        \end{equation*}
        Hence,
        $$\check{J}^N_0(\check{u}_0(\cdot)) \le \check{J}^N_0(\check{u}_0(\cdot) + \varepsilon u_0(\cdot))$$
        if and only if (\ref{stationarity condition of leader}) and (\ref{convexity condition of leader}) hold. The proof is complete.
    \end{proof}

    \begin{Remark}
        Assumption (A4) represents a special instance of the convexity condition articulated in equation (\ref{convexity condition of leader}).
    \end{Remark}

    Under the premise of (A4), we can determine the optimal control of the leader is
    \begin{equation}\label{open-loop optimal strategies of leader}
        \check{u}_0 = -R_0^{-1}B_0^\top \mathbb{E}_0[\check{y}_0].
    \end{equation}
    So the related Hamiltonian system can be represented by
    \begin{equation}\label{Hamiltonian system of leader}
        \left\{
        \begin{aligned}
            d\check{x}_0 = & \left[ A_0\check{x}_0 - B_0R_0^{-1}B_0^\top \mathbb{E}_0[\check{y}_0] + f_0 \right]dt + D_0dW_0, \quad \check{x}_0(0) = \xi_0, \\
            d\check{x}^{(N)} = & \left[ (A - BR^{-1}B^\top P_N) \check{x}^{(N)} - BR^{-1}B^\top (\Pi_N - P_N) \mathbb{E}[\check{x}^{(N)}] + f - BR^{-1}B^\top \check{\phi}_N \right] dt \\
            &+ DdW^{(N)}, \quad \check{x}^{(N)}(0) = \xi^{(N)},\\
            d\check{\phi}_N = &\ -\left[ (A^\top - \Pi_NBR^{-1}B^\top)\check{\phi}_N + \Pi_Nf - \left(I - \frac{\Gamma}{N}\right)^\top Q (\Gamma_1 \mathbb{E}[\check{x}_0] + \eta) \right]dt, \\
            &\ \check{\phi}_N(T) = 0, \\
            d\check{y}_0 = &\ -\left[ A_0^\top \check{y}_0 + \Gamma_1^\top Q \left(I - \frac{\Gamma}{N}\right) \mathbb{E}[\check{\psi}_N] + Q_0(\check{x}_0 - \Gamma_0 \check{x}^{(N)} - \eta_0) \right]dt \\
            &\ + \check{z}_0 dW_0 + \check{z} dW^{(N)}, \quad \check{y}_0(T) = 0, \\
            d\check{y}^{(N)} = &\ -\big[ (A - BR^{-1}B^\top P_N)^\top \check{y}^{(N)} - (\Pi_N - P_N)BR^{-1}B^\top \mathbb{E}[\check{y}^{(N)}] \\
            &\ - \Gamma_0^\top Q_0(\check{x}_0 - \Gamma_0\check{x}^{(N)} - \eta_0) \big]dt + \check{z}_0^{(N)} dW_0 + \check{z}^{(N)} dW^{(N)}, \quad \check{y}^{(N)}(T) = 0, \\
            d\check{\psi}_N = &\ \left[ BR^{-1}B^\top \check{y}^{(N)} + (A - BR^{-1}B^\top \Pi_N)\check{\psi}_N \right]dt, \quad \check{\psi}_N(0) = 0.
        \end{aligned}
        \right.
    \end{equation}

    Owing to the presence of the term $\mathbb{E}_0[\check{y}_0]$ in equation (\ref{open-loop optimal strategies of leader}), it is necessary to derive the equation of $\mathbb{E}_0[\check{y}_0]$.
    Prior to this derivation, we will first present the following material.

    By taking conditional expectations $\mathbb{E}_0[\cdot]$, we have
    \begin{equation}\label{conditional expectation for F0}
        \mathbb{E}_0[\check{x}_0] = \check{x}_0, \quad \mathbb{E}_0[\check{x}^{(N)}] = \mathbb{E}[\check{x}^{(N)}], \quad \mathbb{E}_0[\check{\phi}_N] = \check{\phi}_N.
    \end{equation}

    Setting conditional expectations $\mathbb{E}_0[\cdot]$ for (\ref{Hamiltonian system of leader}), we have
    \begin{equation}\label{conditional Hamiltonian system of leader}
        \left\{
        \begin{aligned}
            d\check{x}_0 = & \left[ A_0\check{x}_0 - B_0R_0^{-1}B_0^\top \mathbb{E}_0[\check{y}_0] + f_0 \right]dt + D_0dW_0, \quad \check{x}_0(0) = \xi_0, \\
            d\mathbb{E}[\check{x}^{(N)}] = & \left[ (A - BR^{-1}B^\top \Pi_N) \mathbb{E}[\check{x}^{(N)}] + f - BR^{-1}B^\top \check{\phi}_N \right] dt, \quad \mathbb{E}_0[\check{x}^{(N)}(0)] = \bar{\xi},\\
            d\check{\phi}_N = &\ -\left[ (A^\top - \Pi_NBR^{-1}B^\top)\check{\phi}_N + \Pi_Nf - \left(I - \frac{\Gamma}{N}\right)^\top Q (\Gamma_1 \mathbb{E}[\check{x}_0] + \eta) \right]dt, \\
            &\ \check{\phi}_N(T) = 0, \\
            d\mathbb{E}_0[\check{y}_0] = &\ -\left[ A_0^\top \mathbb{E}_0[\check{y}_0] + \Gamma_1^\top Q \left(I - \frac{\Gamma}{N}\right) \mathbb{E}[\check{\psi}_N] + Q_0(\check{x}_0 - \Gamma_0 \mathbb{E}[\check{x}^{(N)}] - \eta_0) \right]dt \\
            &\ + \mathbb{E}_0[\check{z}_0] dW_0, \quad \mathbb{E}_0[\check{y}_0(T)] = 0, \\
            d\mathbb{E}_0[\check{y}^{(N)}] = &\ -\big[ (A - BR^{-1}B^\top P_N)^\top \mathbb{E}_0[\check{y}^{(N)}] - (\Pi_N - P_N)BR^{-1}B^\top \mathbb{E}[\check{y}^{(N)}] \\
            &\ - \Gamma_0^\top Q_0(\check{x}_0 - \Gamma_0\mathbb{E}[\check{x}^{(N)}] - \eta_0) \big]dt + \mathbb{E}_0[\check{z}_0^{(N)}] dW_0, \quad \check{y}^{(N)}(T) = 0, \\
            d\mathbb{E}_0[\check{\psi}_N] = &\ \left[ BR^{-1}B^\top \mathbb{E}_0[\check{y}^{(N)}] + (A - BR^{-1}B^\top \Pi_N)\mathbb{E}_0[\check{\psi}_N] \right]dt, \quad \mathbb{E}_0[\check{\psi}_N(0)] = 0,
        \end{aligned}
        \right.
    \end{equation}

    Denote
    \begin{equation*}
        \begin{aligned}
            \check{X} :=&
            \begin{bmatrix}
                \check{x}_0 \\ \mathbb{E}[\check{x}^{(N)}] \\ \mathbb{E}_0[\check{\psi}_N]
            \end{bmatrix}, \quad
            \check{Y} :=
            \begin{bmatrix}
                \mathbb{E}_0[\check{y}_0] \\ \mathbb{E}_0[\check{y}^{(N)}] \\ \check{\phi}_N
            \end{bmatrix}, \quad
            \check{Z} :=
            \begin{bmatrix}
                \mathbb{E}_0[\check{z}_0] \\ \mathbb{E}_0[\check{z}_0^{(N)}] \\ 0
            \end{bmatrix},\\
            \check{\mathcal{D}} :=&
            \begin{bmatrix}
                D_0 \\ 0 \\ 0
            \end{bmatrix},\quad
            \check{\mathfrak{f}} :=
            \begin{bmatrix}
                f_0 \\ f \\ 0
            \end{bmatrix},\quad
            \check{\mathfrak{f}}_1 :=
            \begin{bmatrix}
                Q_0 \eta_0 \\ - \Gamma_0^\top Q_0 \eta_0 \\ (I - \frac{\Gamma}{N})^\top Q \eta - \Pi_N f
            \end{bmatrix},\\
            \check{\mathcal{A}} :=&
            \begin{bmatrix}
                A_0 & 0 & 0 \\
                0 & A - BR^{-1}B^\top \Pi_N & 0 \\
                0 & 0 & A - BR^{-1}B^\top \Pi_N \\
            \end{bmatrix},\\
            \check{\mathcal{B}} :=&
            \begin{bmatrix}
                - B_0R_0^{-1}B_0^\top & 0 & 0 \\
                0 & 0 & -BR^{-1}B^\top \\
                0 & -BR^{-1}B^\top & 0 \\
            \end{bmatrix},\\
            \check{\mathcal{A}}_1 :=&
            \begin{bmatrix}
                - Q_0 & Q_0 & 0 \\
                \Gamma_0^\top Q_0 & -\Gamma_0^\top Q_0 & 0 \\
                0 & 0 & 0 \\
            \end{bmatrix},\quad
            \check{\mathcal{B}}_2 :=
            \begin{bmatrix}
                0 & 0 & 0 \\
                0 & (\Pi_N - P_N)BR^{-1}B^\top & 0 \\
                0 & 0 & 0 \\
            \end{bmatrix},\\
            \check{\mathcal{B}}_1 :=&
            \begin{bmatrix}
                - A_0^\top & 0 & 0 \\
                0 & - A^\top + P_NBR^{-1}B^\top & 0 \\
                0 & 0 & - A + \Pi_NBR^{-1}B^\top \\
            \end{bmatrix},\\
            \check{\mathcal{A}}_2 :=&
            \begin{bmatrix}
                0 & 0 & -\Gamma_1^\top Q \left( I - \frac{\Gamma}{N} \right) \\
                0 & 0 & 0 \\
                \left( I - \frac{\Gamma}{N} \right)^\top Q \Gamma_1 & 0 & 0 \\
            \end{bmatrix}.
        \end{aligned}
    \end{equation*}
    With the above notions, we can rewrite (\ref{conditional Hamiltonian system of leader}) as
    \begin{equation}\label{extended dimensional Hamiltonian system of leader}
        \left\{
        \begin{aligned}
            d\check{X} &= \big[ \check{\mathcal{A}} \check{X} + \check{\mathcal{B}} \check{Y} + \check{\mathfrak{f}} \big] dt + \check{\mathcal{D}} dW_0,\quad \check{X}(0) = [\xi_0^\top, \bar{\xi}^\top, 0]^\top, \\
            d\check{Y} &= \big[ \check{\mathcal{A}}_1 \check{X} + \check{\mathcal{B}}_1 \check{Y} + \check{\mathcal{A}}_2 \mathbb{E}[\check{X}] + \check{\mathcal{B}}_2 \mathbb{E}[\check{Y}] + \check{\mathfrak{f}}_1 \big] dt + \check{Z} dW_0,\quad \check{Y}(T) = 0.
        \end{aligned}
        \right.
    \end{equation}

    Suppose $(\check{X}(\cdot), \check{Y}(\cdot), \check{Z}(\cdot))$ is an adapted solution to (\ref{extended dimensional Hamiltonian system of leader}). We assume that $\check{X}(\cdot)$ and $\check{Y}(\cdot)$ are related by the following affine transformation
    \begin{equation}
        \check{Y}(\cdot) = \check{\mathcal{P}}(\cdot) \check{X}(\cdot) + \check{\mathcal{K}}(\cdot) \mathbb{E}[\check{X}(\cdot)] + \check{\mathcal{V}}(\cdot),
    \end{equation}
    where $\check{\mathcal{P}}(\cdot)$, $\check{\mathcal{K}}(\cdot)$ and $\check{\mathcal{V}}(\cdot)$ are both differentiable functions, with $\check{\mathcal{P}}(T) = 0$, $\check{\mathcal{K}}(T) = 0$ and $\check{\mathcal{V}}(T) = 0$.
    Subsequently, employing It\^{o}'s formula, we arrive at the following
    \begin{equation}\label{Ito formula of extended dimensional Hamiltonian system of leader}
        \begin{aligned}
            d\check{Y} = &\big[ \dot{\check{\mathcal{P}}} \check{X} + \dot{\check{\mathcal{K}}} \mathbb{E}[\check{X}] + \dot{\check{\mathcal{V}}} \big] dt + \check{\mathcal{P}} \big[ \check{\mathcal{A}} \check{X} + \check{\mathcal{B}} (\check{\mathcal{P}} \check{X} + \check{\mathcal{K}} \mathbb{E}[\check{X}] + \check{\mathcal{V}}) + \check{\mathfrak{f}} \big] dt + \check{\mathcal{P}} \check{\mathcal{D}}dW_0 \\
            &+ \check{\mathcal{K}} \big[ \check{\mathcal{A}} \mathbb{E}[\check{X}] + \check{\mathcal{B}} ((\check{\mathcal{P}} + \check{\mathcal{K}}) \mathbb{E}[\check{X}] + \check{\mathcal{V}}) + \check{\mathfrak{f}} \big] dt.
        \end{aligned}
    \end{equation}
    Now, comparing (\ref{Ito formula of extended dimensional Hamiltonian system of leader}) with the backward equation in (\ref{extended dimensional Hamiltonian system of leader}), it follows that
    \begin{equation*}
        \check{Z} = \check{\mathcal{P}} \check{\mathcal{D}},
    \end{equation*}
    \begin{equation}\label{Riccati-1 of leader}
        \dot{\check{\mathcal{P}}} + \check{\mathcal{P}} \check{\mathcal{A}} + \check{\mathcal{P}} \check{\mathcal{B}} \check{\mathcal{P}} - \check{\mathcal{A}}_1 - \check{\mathcal{B}}_1 \check{\mathcal{P}} = 0, \quad \check{\mathcal{P}}(T) = 0,
    \end{equation}
    \begin{equation}\label{Riccati-2 of leader}
        \dot{\check{\mathcal{K}}} + \check{\mathcal{P}}\check{\mathcal{B}}\check{\mathcal{K}} + \check{\mathcal{K}}\check{\mathcal{A}} + \check{\mathcal{K}}\check{\mathcal{B}}(\check{\mathcal{P}}+\check{\mathcal{K}}) - \check{\mathcal{B}}_1 \check{\mathcal{K}} + \check{\mathcal{A}}_1 + \check{\mathcal{B}}_2 (\check{\mathcal{P}}+\check{\mathcal{K}}) = 0, \quad \check{\mathcal{K}}(T) = 0,
    \end{equation}
    \begin{equation}\label{BODE-1 of leader}
        \dot{\check{\mathcal{V}}} + (\check{\mathcal{P}}+\check{\mathcal{K}})\check{\mathcal{B}}(\check{\mathcal{V}}+\check{\mathfrak{f}}) - \check{\mathcal{B}}_1 \check{\mathcal{V}} + \check{\mathcal{B}}_2 \check{\mathcal{V}} + \check{\mathfrak{f}}_1 = 0, \quad \check{\mathcal{V}}(T) = 0.
    \end{equation}

    Let $\check{\mathcal{M}}(\cdot) := \check{\mathcal{P}}(\cdot) + \check{\mathcal{K}}(\cdot)$, then $\check{\mathcal{M}}(\cdot)$ satisfies
    \begin{equation}\label{Riccati-1+2 of leader}
        \dot{\check{\mathcal{M}}} + \check{\mathcal{M}}\check{\mathcal{A}} - \check{\mathcal{B}}_1 \check{\mathcal{M}} + \check{\mathcal{M}}\check{\mathcal{B}}\check{\mathcal{M}} + \check{\mathcal{B}}_2 \check{\mathcal{M}} + \check{\mathcal{A}}_2 -\check{\mathcal{A}}_1 = 0, \quad \check{\mathcal{M}}(T) = 0,
    \end{equation}
    thus (\ref{BODE-1 of leader}) can be rewritten as
    \begin{equation}\label{BODE-2 of leader}
        \dot{\check{\mathcal{V}}} + \check{\mathcal{M}}\check{\mathcal{B}}(\check{\mathcal{V}}+\check{\mathfrak{f}}) - \check{\mathcal{B}}_1 \check{\mathcal{V}} + \check{\mathcal{B}}_2 \check{\mathcal{V}} + \check{\mathfrak{f}}_1 = 0, \quad \check{\mathcal{V}}(T) = 0.
    \end{equation}

    By Theorem 4.1 of \cite{Ma-Yong-99} again, if (\ref{Riccati-1 of leader}) and (\ref{Riccati-2 of leader}) admit a solution $\check{\mathcal{P}}(\cdot)$, $\check{\mathcal{M}}(\cdot)$, respectively, then FBSDE (\ref{extended dimensional Hamiltonian system of leader}) admits a unique adapted solution $(\check{X}(\cdot), \check{Y}(\cdot), \check{Z}(\cdot))$.

    \begin{Remark}
        Note that the Riccati equations (\ref{Riccati-1 of leader}) and (\ref{Riccati-1+2 of leader}) are nonsymmetric, hence in general, they may not admit a solution.
        We give an existing conclusion for the existence of a unique solution to (\ref{Riccati-1 of leader}) and (\ref{Riccati-1+2 of leader}) (\cite{Ma-Yong-99}).
        If det$\left\{ [0,I] e^{\mathbb{A}t}
        \begin{bmatrix}
            0 \\ I
        \end{bmatrix} \right\} > 0$ with $\mathbb{A} := \begin{bmatrix}
            \check{\mathcal{A}} & \check{\mathcal{B}} \\
            \check{\mathcal{A}}_1 & \check{\mathcal{B}}_2 \\
        \end{bmatrix}$, then (\ref{Riccati-1 of leader}) admits a unique solution
        $$\check{\mathcal{P}} = - \left[ [0,I] e^{\mathbb{A}(T-t)}
        \begin{bmatrix}
            0 \\ I
        \end{bmatrix}
        \right]^{-1} [0,I] e^{\mathbb{A}(T-t)}
        \begin{bmatrix}
            I \\ 0
        \end{bmatrix}.$$
        (\ref{Riccati-1+2 of leader}) admit a solution $\check{\mathcal{M}}(\cdot)$, which is also the same as before.
    \end{Remark}

    We summarize the above analysis in the following theorem.

    \begin{mythm}\label{leader thm2}
        Under Assumptions (A1)-(A4), let the followers adopt the optimal strategy (\ref{optimal strategy of followers}), if (\ref{Riccati-1 of leader}) and (\ref{Riccati-2 of leader}) admit a solution $\check{\mathcal{P}}(\cdot)$, $\check{\mathcal{M}}(\cdot)$, respectively, then Problem {\bf (PG2)} admits a unique solution
        \begin{equation}\label{optimal strategy of leader}
            \check{u}_0 = -R_0^{-1} B_0^\top e_1 \left( \check{\mathcal{P}} \check{X} + \check{\mathcal{K}} \mathbb{E}[\check{X}] + \check{\mathcal{V}} \right),
        \end{equation}
        where $e_1 = [I,0,0]$,
        $$d\check{X} = \left[ (\check{\mathcal{A}} + \check{\mathcal{B}} \check{\mathcal{P}}) \check{X} + \check{\mathcal{B}} \check{\mathcal{K}} \mathbb{E}[\check{X}] + \check{\mathcal{B}} \check{\mathcal{V}} + \check{\mathfrak{f}} \right] dt + \check{\mathcal{D}} dW_0,\quad \check{X}(0) = [\xi_0^\top, \bar{\xi}^\top, 0]^\top, $$
        $$d\mathbb{E}[\check{X}] = \left[(\check{\mathcal{A}} + \check{\mathcal{B}} \check{\mathcal{M}}) \mathbb{E}[\check{X}] + \check{\mathcal{B}} \check{\mathcal{V}} + \check{\mathfrak{f}} \right] dt, \quad \mathbb{E}[\check{X}(0)] = [\bar{\xi}_0^\top, \bar{\xi}^\top, 0].$$
    \end{mythm}

    Combining Theorem \ref{followers thm2} and Theorem \ref{leader thm2}, strategy $(\check{u}_0,\check{u}_1,\cdots,\check{u}_N)$ that satisfies equations (\ref{optimal strategy of followers}) and (\ref{optimal strategy of leader}) constitutes a decentralized Stackelberg-Nash equilibrium.

    \section{LQ Stackelberg MF teams}

    \subsection{MF teams for the $N$ followers}

    Consistent with Subsection 3.1, we first address the multiagent team optimal problem for $N$ followers.

    \noindent {\bf (PS1)}: Minimize social cost $J^N_{soc}$ over $u_i(\cdot) \in \mathscr{U}_{d,i}[0,T],i = 1,\cdots,N$.

    For convenience, we define the following notation:
    $$Q_{\Gamma} := \Gamma^\top Q + Q \Gamma - \Gamma^\top Q \Gamma, \quad Q_{\Gamma_1} := Q \Gamma_1 - \Gamma^\top Q \Gamma_1, \quad Q_{\eta} := Q \eta - \Gamma^\top Q \eta.$$

    We arrive at the subsequent finding.

    \begin{mythm}\label{team followers thm1}
        Under Assumption (A1)-(A2), let $u_0(\cdot) \in \mathscr{U}_{d,0}[0,T]$ be given, for the initial value $\xi_i, i = 1,\cdots,N$, {\bf (PS1)} admits an optimal control $\hat{u}_i(\cdot) \in \mathscr{U}_{d,i}[0,T], i = 1,\cdots,N$, if and only if the following two conditions hold:

        (i) The adapted solution $\left( \hat{x}_i(\cdot), \hat{p}_i(\cdot), \hat{q}^j_i(\cdot), i = 1,\cdots,N, j = 0,1,\cdots,N \right)$ to the FBSDE
        \begin{equation}\label{adjoint FBSDE of team followers}
            \left\{
            \begin{aligned}
                d\hat{x}_i &= \big[A\hat{x}_i + B\hat{u}_i + f\big]dt + DdW_i,\\
                d\hat{p}_i &= -\left[A^\top \hat{p}_i + Q\hat{x}_i - Q_{\Gamma} \hat{x}^{(N)} - Q_{\Gamma_1} x_0 - Q_{\eta} \right]dt + \sum_{j=0}^{N} \hat{q}_i^j dW_j,\quad t\in[0,T],\\
                \hat{x}_i(0) &= \xi_i ,\quad \hat{p}_i(T) = 0 , \quad i = 1,\cdots,N,
            \end{aligned}
            \right.
        \end{equation}
        satisfies the following stationarity condition:
        \begin{equation}\label{stationarity condition of team followers}
            B^\top \mathbb{E}_i[\hat{p}_i] + R\hat{u}_i = 0, \quad i = 1,\cdots,N.
        \end{equation}

        (ii) For $i = 1,\cdots,N$, the following convexity condition holds:
        \begin{equation}\label{convexity condition of team followers}
            \sum_{i=1}^{N} \mathbb{E} \int_0^T \left[ ||\tilde{x}_i - \Gamma \tilde{x}^{(N)}||^2_Q + ||u_i||^2_R \right]dt \ge 0, \quad \forall u_i(\cdot) \in \mathscr{U}_{d,i}[0,T],\ i = 1,\cdots,N,
        \end{equation}
        where $\tilde{x}_i(\cdot)$ is the solution to the following RDE:
        \begin{equation}\label{random differential equation of team followers}
            \left\{
            \begin{aligned}
                &d\tilde{x}_i = [A\tilde{x}_i + Bu_i]dt,\\
                &\tilde{x}_i(0) = 0,
            \end{aligned}
            \right.
        \end{equation}
        with $\tilde{x}^{(N)} := \frac{1}{N} \sum_{i=1}^{N} \tilde{x}_i.$
    \end{mythm}

    Under (A3), the set of open-loop optimal strategies for followers is
    \begin{equation}\label{open-loop optimal strategies of team followers}
        \hat{u}_i = -R^{-1} B^\top \mathbb{E}_i[\hat{p}_i],\quad i = 1,\cdots,N.
    \end{equation}

    Setting conditional expectations $\mathbb{E}_i[\cdot]$ for $\hat{p}_i$, we get
    \begin{equation}\label{pi conditional expectation for team Fi}
        \begin{aligned}
            d\mathbb{E}_i[\hat{p}_i] = &-\left[A^\top \mathbb{E}_i[\hat{p}_i] + \left( Q - \frac{Q_\Gamma}{N} \right) \hat{x}_i - \frac{N-1}{N} Q_\Gamma \mathbb{E}[\hat{x}_i] - Q_{\Gamma_1} \mathbb{E}[x_0] - Q_{\eta} \right]dt \\
            &+ \mathbb{E}_i[\hat{q}_i^i] dW_i, \quad \mathbb{E}_i[\hat{p}_i(T)] = 0.
    \end{aligned}
    \end{equation}

    We consider the transformation
    \begin{equation}\label{transformation of team adjoint equation}
        \mathbb{E}_i[\hat{p}_i(\cdot)] = \hat{P}_N(\cdot)\hat{x}_i(\cdot) + \hat{K}_N(\cdot)\mathbb{E}[\hat{x}_i](\cdot) + \hat{\phi}_N(\cdot), \quad i = 1,\cdots,N,
    \end{equation}
    where $\hat{P}_N(\cdot)$, $\hat{K}_N(\cdot)$ and $\hat{\phi}_N(\cdot)$ are differential functions with $\hat{P}_N(T) = 0$, $\hat{K}_N(T) = 0$ and $\hat{\phi}_N(T) = 0$.
    Applying It\^{o}'s formula, and comparing the coefficients, we have
    $$\mathbb{E}_i[\hat{q}_i^i] = \hat{P}_ND,$$
    \begin{equation}\label{Riccati-1 of team followers}
        \dot{\hat{P}}_N + A^\top \hat{P}_N + \hat{P}_NA - \hat{P}_NBR^{-1}B^\top \hat{P}_N + Q - \frac{Q_\Gamma}{N} = 0,\quad \hat{P}_N(T) = 0,
    \end{equation}
    \begin{equation}\label{Riccati-2 of team followers}
        \begin{aligned}
            &\dot{\hat{K}}_N + A^\top \hat{K}_N + \hat{K}_NA - \hat{P}_NBR^{-1}B^\top \hat{K}_N - \hat{K}_NBR^{-1}B^\top (\hat{P}_N + \hat{K}_N) \\
            &\qquad - \frac{N-1}{N} Q_\Gamma = 0,\quad \hat{K}_N(T) = 0,
        \end{aligned}
    \end{equation}
    \begin{equation}\label{BODE-1 of team followers}
        \dot{\hat{\phi}}_N + \left[A^\top - (\hat{P}_N + \hat{K}_N)BR^{-1}B^\top \right] \hat{\phi}_N + (\hat{P}_N + \hat{K}_N)f - Q_{\Gamma_1} \mathbb{E}[x_0] - Q_{\eta} = 0,\quad \hat{\phi}_N(T) = 0.
    \end{equation}

    Let $\hat{\Pi}_N(\cdot) := \hat{P}_N(\cdot) + \hat{K}_N(\cdot)$, then $\hat{\Pi}_N(\cdot)$ satisfies
    \begin{equation}\label{Riccati-1+2 of team followers}
        \dot{\hat{\Pi}}_N + A^\top \hat{\Pi}_N + \hat{\Pi}_NA - \hat{\Pi}_NBR^{-1}B^\top \hat{\Pi}_N + Q - Q_\Gamma = 0,\quad \hat{\Pi}_N(T) = 0,
    \end{equation}
    thus (\ref{BODE-1 of team followers}) can be rewritten as
    \begin{equation}\label{BODE-2 of team followers}
        \dot{\hat{\phi}}_N + \left[A^\top - \hat{\Pi}_N BR^{-1}B^\top \right] \hat{\phi}_N + \hat{\Pi}_N f - Q_{\Gamma_1} \mathbb{E}[x_0] - Q_{\eta} = 0,\quad \hat{\phi}_N(T) = 0.
    \end{equation}

    \begin{Remark}
        Observe that (\ref{Riccati-1 of team followers}) is a symmetric Riccati differential equation. It is noted that if  (A3) is satisfied, a unique solution exists, given that
        $$Q - \frac{Q_\Gamma}{N} = \frac{N-1}{N}Q + \frac{1}{N} (I - \Gamma)^\top Q(I - \Gamma).$$
        Similarly, (\ref{Riccati-1+2 of team followers}) is a symmetric Riccati differential equation, if it satisfies (A3), then it admits a unique solution.
        Thus (\ref{Riccati-2 of team followers}) admits a unique solution.
        And then (\ref{BODE-2 of team followers}) admits a unique solution.
    \end{Remark}

    \begin{mythm}\label{team followers thm2}
        Under Assumptions (A1)-(A3), for given $u_0(\cdot) \in \mathscr{U}_{d,0}[0,T]$, Problem {\bf (PS1)} admits a unique solution
        \begin{equation}\label{optimal strategy of team followers}
            \hat{u}_i = -R^{-1} B^\top \left( \hat{P}_N\hat{x}_i + \hat{K}_N\mathbb{E}[\hat{x}_i] + \hat{\phi}_N \right), \quad i = 1,\cdots,N,
        \end{equation}
        where
        $$d\hat{x}_i = \left[(A - BR^{-1}B^\top \hat{P}_N) \hat{x}_i - BR^{-1}B^\top \hat{K}_N \mathbb{E}[\hat{x}_i] + f - BR^{-1}B^\top \hat{\phi}_N \right] dt + DdW_i, \quad \hat{x}_i(0) = \xi_i,$$
        $$d\mathbb{E}[\hat{x}_i] = \left[(A - BR^{-1}B^\top \hat{\Pi}_N) \mathbb{E}[\hat{x}_i] + f - BR^{-1}B^\top \hat{\phi}_N \right] dt, \quad \mathbb{E}[\hat{x}_i(0)] = \bar{\xi}.$$
    \end{mythm}

    \subsection{Optimal strategy of the leader}

    After the followers have implemented their strategies $\hat{u}_i(\cdot), i = 1,\cdots,N$ according to (\ref{optimal strategy of team followers}), we proceed to investigate an optimal control problem for the leader.

    {\bf (PS2)}: Minimize $\hat{J}^N_0(u_0(\cdot))$ over $u_0(\cdot) \in \mathscr{U}_{d,0}[0,T]$, where
    \begin{equation}\label{cost of team leader}
        \hat{J}^N_0(u_0(\cdot)) := \frac{1}{2} \mathbb{E} \int_0^T \left[||x_0(t) - \Gamma_0\hat{x}^{(N)}(t) - \eta_0||^2_{Q_0} + ||u_0(t)||^2_{R_0}\right]dt,
    \end{equation}
    subject to
    \begin{equation}\label{state of team leader}
        \left\{
        \begin{aligned}
            dx_0 =& \left[A_0x_0 + B_0u_0 + f_0\right]dt + D_0dW_0,\\
            d\hat{x}_i =& \left[(A - BR^{-1}B^\top \hat{P}_N) \hat{x}_i - BR^{-1}B^\top \hat{K}_N \mathbb{E}[\hat{x}_i] + f - BR^{-1}B^\top \hat{\phi}_N \right] dt\\
            &\ + DdW_i, \quad \hat{x}_i(0) = \xi_i,\\
            d\hat{\phi}_N =& - \left[\left(A^\top - \hat{\Pi}_N BR^{-1}B^\top \right) \check{\phi}_N + \hat{\Pi}_N f - Q_{\Gamma_1} \mathbb{E}[x_0] - Q_{\eta} \right] dt,\\
            x_0(0) =&\ \xi_0, \quad \hat{x}_i(0) = \xi_i,\quad \hat{\phi}_N(T) = 0, \quad i = 1,\cdots,N.
        \end{aligned}
        \right.
    \end{equation}

    We have the following result.

    \begin{mythm}\label{team leader thm1}
        Under Assumptions (A1)-(A3), let the followers adopt the optimal strategy (\ref{optimal strategy of team followers}). Then Problem {\bf (PS2)} admits an optimal control $\hat{u}_0(\cdot) \in \mathscr{U}_{d,0}[0,T]$, if and only if the following two conditions hold:

        (i) The adapted solution $\left(\hat{x}_0(\cdot), \hat{x}^{(N)}(\cdot), \hat{\phi}_N(\cdot), \hat{y}_0(\cdot), \hat{z}_0(\cdot), \hat{z}(\cdot), \hat{y}^{(N)}(\cdot), \hat{z}_0^{(N)}(\cdot), \hat{z}^{(N)}(\cdot), \hat{\psi}_N(\cdot) \right)$ to the FBSDE
        \begin{equation}\label{adjoint FBSDE of team leader}
            \left\{
            \begin{aligned}
                d\hat{x}_0 = & \left[ A_0\hat{x}_0 + B_0\hat{u}_0 + f_0 \right]dt + D_0dW_0, \quad \hat{x}_0(0) = \xi_0, \\
                d\hat{x}^{(N)} = & \left[ (A - BR^{-1}B^\top \hat{P}_N) \hat{x}^{(N)} - BR^{-1}B^\top (\hat{\Pi}_N - \hat{P}_N) \mathbb{E}[\hat{x}^{(N)}] + f - BR^{-1}B^\top \hat{\phi}_N \right] dt \\
                &+ DdW^{(N)}, \quad \hat{x}^{(N)}(0) = \xi^{(N)},\\
                d\hat{\phi}_N = &\ -\left[ (A^\top - \hat{\Pi}_NBR^{-1}B^\top)\hat{\phi}_N + \hat{\Pi}_Nf - Q_{\Gamma_1} \mathbb{E}[\hat{x}_0] - Q_{\eta} \right]dt, \quad \hat{\phi}_N(T) = 0, \\
                d\hat{y}_0 = &\ -\left[ A_0^\top \hat{y}_0 + Q_{\Gamma_1}^\top \mathbb{E}[\hat{\psi}_N] + Q_0(\hat{x}_0 - \Gamma_0 \hat{x}^{(N)} - \eta_0) \right]dt \\
                &\ + \hat{z}_0 dW_0 + \hat{z} dW^{(N)}, \quad \hat{y}_0(T) = 0, \\
                d\hat{y}^{(N)} = &\ -\big[ (A - BR^{-1}B^\top \hat{P}_N)^\top \hat{y}^{(N)} - (\hat{\Pi}_N - \hat{P}_N)BR^{-1}B^\top \mathbb{E}[\hat{y}^{(N)}] \\
                &\ - \Gamma_0^\top Q_0(\hat{x}_0 - \Gamma_0\hat{x}^{(N)} - \eta_0) \big]dt + \hat{z}_0^{(N)} dW_0 + \hat{z}^{(N)} dW^{(N)}, \quad \hat{y}^{(N)}(T) = 0, \\
                d\hat{\psi}_N = &\ \left[ BR^{-1}B^\top \hat{y}^{(N)} + (A - BR^{-1}B^\top \hat{\Pi}_N)\hat{\psi}_N \right]dt, \quad \hat{\psi}_N(0) = 0,
            \end{aligned}
            \right.
        \end{equation}
        satisfies the following stationarity condition:
        \begin{equation}\label{stationarity condition of team leader}
            B_0^\top \mathbb{E}_0[\hat{y}_0] + R_0\hat{u}_0 = 0.
        \end{equation}

        (ii) The following convexity condition holds:
        \begin{equation}\label{convexity condition of team leader}
            \mathbb{E} \int_0^T \left[ ||\tilde{x}_0 - \Gamma_0\tilde{x}^{(N)}||_{Q_0}^2 + ||u_0||_{R_0}^2\right]dt \ge 0, \quad \forall u_0(\cdot) \in \mathscr{U}_{d,0}[0,T],
        \end{equation}
        where $(\tilde{x}_0(\cdot), \tilde{x}^{(N)}(\cdot), \tilde{\phi}_N(\cdot))$ is the solution to the following FBSDE
        \begin{equation}\label{random FBSDE of team leader}
            \left\{
            \begin{aligned}
                d\tilde{x}_0 = &\ \left[ A_0\tilde{x}_0 + B_0u_0 \right]dt, \quad \tilde{x}_0(0) = 0, \\
                d\tilde{x}^{(N)} = &\ \left[ (A - BR^{-1}B^\top P_N)\tilde{x}^{(N)} - BR^{-1}B^\top (\Pi_N - P_N) \mathbb{E}[\tilde{x}^{(N)}] - BR^{-1}B^\top \tilde{\phi}_N \right]dt,\\
                & \ \tilde{x}^{(N)}(0) = 0, \\
                d\tilde{\phi}_N = & -\left[(A^\top - \Pi_NBR^{-1}B^\top)\tilde{\phi}_N - Q_{\Gamma_1} \mathbb{E}[\tilde{x}_0] \right]dt, \quad \tilde{\phi}_N(T) = 0.
            \end{aligned}
            \right.
        \end{equation}
    \end{mythm}

    Under the premise of (A4), we are able to deduce that the optimal control for the leader is
    \begin{equation}\label{open-loop optimal strategies of team leader}
        \hat{u}_0 = -R_0^{-1}B_0^\top \mathbb{E}_0[\hat{y}_0].
    \end{equation}
    Hence, the associated Hamiltonian system can be expressed as
    \begin{equation}\label{Hamiltonian system of team leader}
        \left\{
        \begin{aligned}
            d\hat{x}_0 = & \left[ A_0\hat{x}_0 - B_0R_0^{-1}B_0^\top \mathbb{E}_0[\hat{y}_0] + f_0 \right]dt + D_0dW_0, \quad \hat{x}_0(0) = \xi_0, \\
            d\hat{x}^{(N)} = & \left[ (A - BR^{-1}B^\top \hat{P}_N) \hat{x}^{(N)} - BR^{-1}B^\top (\hat{\Pi}_N - \hat{P}_N) \mathbb{E}[\hat{x}^{(N)}] + f - BR^{-1}B^\top \hat{\phi}_N \right] dt \\
            &+ DdW^{(N)}, \quad \hat{x}^{(N)}(0) = \xi^{(N)},\\
            d\hat{\phi}_N = &\ -\left[ (A^\top - \hat{\Pi}_NBR^{-1}B^\top)\hat{\phi}_N + \hat{\Pi}_Nf - Q_{\Gamma_1} \mathbb{E}[\hat{x}_0] - Q_{\eta} \right]dt, \quad \hat{\phi}_N(T) = 0, \\
            d\hat{y}_0 = &\ -\left[ A_0^\top \hat{y}_0 + Q_{\Gamma_1}^\top \mathbb{E}[\hat{\psi}_N] + Q_0(\hat{x}_0 - \Gamma_0 \hat{x}^{(N)} - \eta_0) \right]dt \\
            &\ + \hat{z}_0 dW_0 + \hat{z} dW^{(N)}, \quad \hat{y}_0(T) = 0, \\
            d\hat{y}^{(N)} = &\ -\big[ (A - BR^{-1}B^\top \hat{P}_N)^\top \hat{y}^{(N)} - (\hat{\Pi}_N - \hat{P}_N)BR^{-1}B^\top \mathbb{E}[\hat{y}^{(N)}] \\
            &\ - \Gamma_0^\top Q_0(\hat{x}_0 - \Gamma_0\hat{x}^{(N)} - \eta_0) \big]dt + \hat{z}_0^{(N)} dW_0 + \hat{z}^{(N)} dW^{(N)}, \quad \hat{y}^{(N)}(T) = 0, \\
            d\hat{\psi}_N = &\ \left[ BR^{-1}B^\top \hat{y}^{(N)} + (A - BR^{-1}B^\top \hat{\Pi}_N)\hat{\psi}_N \right]dt, \quad \hat{\psi}_N(0) = 0.
        \end{aligned}
        \right.
    \end{equation}

    Because of the appearance of the $\mathbb{E}_0[\hat{y}_0]$ term in (\ref{open-loop optimal strategies of team leader}), we must derive the equation for $\mathbb{E}_0[\hat{y}_0]$.
    Set conditional expectations $\mathbb{E}_0[\cdot]$ for (\ref{Hamiltonian system of team leader}), then we have
    \begin{equation}\label{conditional Hamiltonian system of team leader}
        \left\{
        \begin{aligned}
            d\hat{x}_0 = & \left[ A_0\hat{x}_0 - B_0R_0^{-1}B_0^\top \mathbb{E}_0[\hat{y}_0] + f_0 \right]dt + D_0dW_0, \quad \hat{x}_0(0) = \xi_0, \\
            d\mathbb{E}[\hat{x}^{(N)}] = & \left[ (A - BR^{-1}B^\top \hat{\Pi}_N) \mathbb{E}[\hat{x}^{(N)}] + f - BR^{-1}B^\top \hat{\phi}_N \right] dt, \quad \mathbb{E}_0[\hat{x}^{(N)}(0)] = \bar{\xi},\\
            d\hat{\phi}_N = &\ -\left[ (A^\top - \hat{\Pi}_NBR^{-1}B^\top)\hat{\phi}_N + \hat{\Pi}_Nf - Q_{\Gamma_1} \mathbb{E}[\hat{x}_0] - Q_{\eta} \right]dt, \quad \hat{\phi}_N(T) = 0, \\
            d\mathbb{E}_0[\hat{y}_0] = &\ -\left[ A_0^\top \mathbb{E}_0[\hat{y}_0] + Q_{\Gamma_1}^\top \mathbb{E}[\hat{\psi}_N] + Q_0(\hat{x}_0 - \Gamma_0 \mathbb{E}[\hat{x}^{(N)}] - \eta_0) \right]dt \\
            &\ + \mathbb{E}_0[\hat{z}_0] dW_0, \quad \mathbb{E}_0[\hat{y}_0(T)] = 0, \\
            d\mathbb{E}_0[\hat{y}^{(N)}] = &\ -\big[ (A - BR^{-1}B^\top \hat{P}_N)^\top \mathbb{E}_0[\hat{y}^{(N)}] - (\hat{\Pi}_N - \hat{P}_N)BR^{-1}B^\top \mathbb{E}[\hat{y}^{(N)}] \\
            &\ - \Gamma_0^\top Q_0(\hat{x}_0 - \Gamma_0\mathbb{E}[\hat{x}^{(N)}] - \eta_0) \big]dt + \mathbb{E}[\hat{z}_0^{(N)}] dW_0, \quad \hat{y}^{(N)}(T) = 0, \\
            d\mathbb{E}_0[\hat{\psi}_N] = &\ \left[ BR^{-1}B^\top \mathbb{E}_0[\hat{y}^{(N)}] + (A - BR^{-1}B^\top \hat{\Pi}_N)\mathbb{E}_0[\hat{\psi}_N] \right]dt, \quad \mathbb{E}_0[\hat{\psi}_N(0)] = 0,
        \end{aligned}
        \right.
    \end{equation}

    Denote
    \begin{equation*}
        \begin{aligned}
            \hat{X} :=&
            \begin{bmatrix}
                \hat{x}_0 \\ \mathbb{E}[\hat{x}^{(N)}] \\ \mathbb{E}_0[\hat{\psi}_N]
            \end{bmatrix}, \quad
            \hat{Y} :=
            \begin{bmatrix}
                \mathbb{E}_0[\hat{y}_0] \\ \mathbb{E}_0[\hat{y}^{(N)}] \\ \hat{\phi}_N
            \end{bmatrix}, \quad
            \hat{Z} :=
            \begin{bmatrix}
                \mathbb{E}_0[\hat{z}_0] \\ \mathbb{E}_0[\hat{z}_0^{(N)}] \\ 0
            \end{bmatrix},\\
            \hat{\mathcal{D}} :=&
            \begin{bmatrix}
                D_0 \\ 0 \\ 0
            \end{bmatrix},\quad
            \hat{\mathfrak{f}} :=
            \begin{bmatrix}
                f_0 \\ f \\ 0
            \end{bmatrix},\quad
            \hat{\mathfrak{f}}_1 :=
            \begin{bmatrix}
                Q_0 \eta_0 \\ - \Gamma_0^\top Q_0 \eta_0 \\ Q_{\eta} - \hat{\Pi}_N f
            \end{bmatrix},\\
            \hat{\mathcal{A}} :=&
            \begin{bmatrix}
                A_0 & 0 & 0 \\
                0 & A - BR^{-1}B^\top \hat{\Pi}_N & 0 \\
                0 & 0 & A - BR^{-1}B^\top \hat{\Pi}_N \\
            \end{bmatrix},\\
         \end{aligned}
    \end{equation*}
    \begin{equation*}
        \begin{aligned}
            \hat{\mathcal{B}} :=&
            \begin{bmatrix}
                - B_0R_0^{-1}B_0^\top & 0 & 0 \\
                0 & 0 & -BR^{-1}B^\top \\
                0 & -BR^{-1}B^\top & 0 \\
            \end{bmatrix},\\
            \hat{\mathcal{A}}_1 :=&
            \begin{bmatrix}
                - Q_0 & Q_0 & 0 \\
                \Gamma_0^\top Q_0 & -\Gamma_0^\top Q_0 & 0 \\
                0 & 0 & 0 \\
            \end{bmatrix},\quad
            \hat{\mathcal{A}}_2 :=
            \begin{bmatrix}
                0 & 0 & -Q_{\Gamma_1}^\top \\
                0 & 0 & 0 \\
                Q_{\Gamma_1} & 0 & 0 \\
            \end{bmatrix},\\
            \hat{\mathcal{B}}_1 :=&
            \begin{bmatrix}
                - A_0^\top & 0 & 0 \\
                0 & - A^\top + \hat{P}_NBR^{-1}B^\top & 0 \\
                0 & 0 & - A + \hat{\Pi}_NBR^{-1}B^\top \\
            \end{bmatrix},\\
            \hat{\mathcal{B}}_2 :=&
            \begin{bmatrix}
                0 & 0 & 0 \\
                0 & (\hat{\Pi}_N - \hat{P}_N)BR^{-1}B^\top & 0 \\
                0 & 0 & 0 \\
            \end{bmatrix},
        \end{aligned}
    \end{equation*}

    With the above notions, we can rewrite (\ref{conditional Hamiltonian system of team leader}) as
    \begin{equation}\label{extended dimensional Hamiltonian system of team leader}
        \left\{
        \begin{aligned}
            d\hat{X} &= \big[ \hat{\mathcal{A}} \hat{X} + \hat{\mathcal{B}} \hat{Y} + \hat{\mathfrak{f}} \big] dt + \hat{\mathcal{D}} dW_0,\quad \hat{X}(0) = [\xi_0^\top, \bar{\xi}^\top, 0]^\top, \\
            d\hat{Y} &= \big[ \hat{\mathcal{A}}_1 \hat{X} + \hat{\mathcal{B}}_1 \hat{Y} + \hat{\mathcal{A}}_2 \mathbb{E}[\hat{X}] + \hat{\mathcal{B}}_2 \mathbb{E}[\hat{Y}] + \hat{\mathfrak{f}}_1 \big] dt + \hat{Z} dW_0,\quad \hat{Y}(T) = 0.
        \end{aligned}
        \right.
    \end{equation}

    Suppose $(\hat{X}(\cdot), \hat{Y}(\cdot), \hat{Z}(\cdot))$ is an adapted solution to (\ref{extended dimensional Hamiltonian system of team leader}). We assume that $\hat{X}(\cdot)$ and $\hat{Y}(\cdot)$ are related by the following affine transformation
    \begin{equation}
        \hat{Y}(\cdot) = \hat{\mathcal{P}}(\cdot) \hat{X}(\cdot) + \hat{\mathcal{K}}(\cdot) \mathbb{E}[\hat{X}(\cdot)] + \hat{\mathcal{V}}(\cdot),
    \end{equation}
    where $\hat{\mathcal{P}}(\cdot)$, $\hat{\mathcal{K}}(\cdot)$ and $\hat{\mathcal{V}}(\cdot)$ are both differentiable functions, with $\hat{\mathcal{P}}(T) = 0$, $\hat{\mathcal{K}}(T) = 0$ and $\hat{\mathcal{V}}(T) = 0$.
    Next, by It\^{o}'s formula, we have
    \begin{equation}\label{Ito formula of extended dimensional Hamiltonian system of team leader}
        \begin{aligned}
            d\hat{Y} = &\big[ \dot{\hat{\mathcal{P}}} \hat{X} + \dot{\hat{\mathcal{K}}} \mathbb{E}[\hat{X}] + \dot{\hat{\mathcal{V}}} \big] dt + \hat{\mathcal{P}} \big[ \hat{\mathcal{A}} \hat{X} + \hat{\mathcal{B}} (\hat{\mathcal{P}} \hat{X} + \hat{\mathcal{K}} \mathbb{E}[\hat{X}] + \hat{\mathcal{V}}) + \hat{\mathfrak{f}} \big] dt + \hat{\mathcal{P}} \hat{\mathcal{D}}dW_0 \\
            &+ \hat{\mathcal{K}} \big[ \hat{\mathcal{A}} \mathbb{E}[\hat{X}] + \hat{\mathcal{B}} ((\hat{\mathcal{P}} + \hat{\mathcal{K}}) \mathbb{E}[\hat{X}] + \hat{\mathcal{V}}) + \hat{\mathfrak{f}} \big] dt.
        \end{aligned}
    \end{equation}
    Now, comparing (\ref{Ito formula of extended dimensional Hamiltonian system of team leader}) with the secend equation in (\ref{extended dimensional Hamiltonian system of team leader}), it follows that
    \begin{equation*}
        \hat{Z} = \hat{\mathcal{P}} \hat{\mathcal{D}},
    \end{equation*}
    \begin{equation}\label{Riccati-1 of team leader}
        \dot{\hat{\mathcal{P}}} + \hat{\mathcal{P}} \hat{\mathcal{A}} + \hat{\mathcal{P}} \hat{\mathcal{B}} \hat{\mathcal{P}} - \hat{\mathcal{A}}_1 - \hat{\mathcal{B}}_1 \hat{\mathcal{P}} = 0, \quad \hat{\mathcal{P}}(T) = 0,
    \end{equation}
    \begin{equation}\label{Riccati-2 of team leader}
        \dot{\hat{\mathcal{K}}} + \hat{\mathcal{P}}\hat{\mathcal{B}}\hat{\mathcal{K}} + \hat{\mathcal{K}}\hat{\mathcal{A}} + \hat{\mathcal{K}}\hat{\mathcal{B}}(\hat{\mathcal{P}}+\hat{\mathcal{K}}) - \hat{\mathcal{B}}_1 \hat{\mathcal{K}} + \hat{\mathcal{A}}_2 + \hat{\mathcal{B}}_2(\hat{\mathcal{P}}+\hat{\mathcal{K}}) = 0, \quad \hat{\mathcal{K}}(T) = 0,
    \end{equation}
    \begin{equation}\label{BODE-1 of team leader}
        \dot{\hat{\mathcal{V}}} + (\hat{\mathcal{P}}+\hat{\mathcal{K}})\hat{\mathcal{B}}(\hat{\mathcal{V}}+\hat{\mathfrak{f}}) - \hat{\mathcal{B}}_1 \hat{\mathcal{V}} + \hat{\mathcal{B}}_2 \hat{\mathcal{V}} + \hat{\mathfrak{f}}_1 = 0, \quad \hat{\mathcal{V}}(T) = 0.
    \end{equation}

    Let $\hat{\mathcal{M}}(\cdot) := \hat{\mathcal{P}}(\cdot) + \hat{\mathcal{K}}(\cdot)$, then $\hat{\mathcal{M}}(\cdot)$ satisfies
    \begin{equation}\label{Riccati-1+2 of team leader}
        \dot{\hat{\mathcal{M}}} + \hat{\mathcal{M}}\hat{\mathcal{A}} - \hat{\mathcal{B}}_1 \hat{\mathcal{M}} + \hat{\mathcal{M}}\hat{\mathcal{B}}\hat{\mathcal{M}} + \hat{\mathcal{B}}_2 \hat{\mathcal{M}} + \hat{\mathcal{A}}_2 - \hat{\mathcal{A}}_1 = 0, \quad \hat{\mathcal{M}}(T) = 0,
    \end{equation}
    thus (\ref{BODE-1 of team leader}) can be rewritten as
    \begin{equation}\label{BODE-2 of team leader}
        \dot{\hat{\mathcal{V}}} + \hat{\mathcal{M}}\hat{\mathcal{B}}(\hat{\mathcal{V}}+\hat{\mathfrak{f}}) - \hat{\mathcal{B}}_1 \hat{\mathcal{V}} + \hat{\mathcal{B}}_2 \hat{\mathcal{V}} + \hat{\mathfrak{f}}_1 = 0, \quad \hat{\mathcal{V}}(T) = 0.
    \end{equation}

    \begin{mythm}\label{team leader thm2}
        Under Assumptions (A1)-(A4), let the followers adopt the optimal strategy (\ref{optimal strategy of team followers}), if (\ref{Riccati-1 of team leader}) and (\ref{Riccati-2 of team leader}) admit a solution $\hat{\mathcal{P}}(\cdot)$, $\hat{\mathcal{M}}(\cdot)$, respectively, then Problem {\bf (PS2)} admits a unique solution
        \begin{equation}\label{optimal strategy of team leader}
            \hat{u}_0 = -R_0^{-1} B_0^\top e_1 \left( \hat{\mathcal{P}} \hat{X} + \hat{\mathcal{K}} \mathbb{E}[\hat{X}] + \hat{\mathcal{V}} \right),
        \end{equation}
        where
        $$d\hat{X} = \left[ (\hat{\mathcal{A}} + \hat{\mathcal{B}} \hat{\mathcal{P}}) \hat{X} + \hat{\mathcal{B}} \hat{\mathcal{K}} \mathbb{E}[\hat{X}] + \hat{\mathcal{B}} \hat{\mathcal{V}} + \hat{\mathfrak{f}} \right] dt + \hat{\mathcal{D}} dW_0,\quad \hat{X}(0) = [\xi_0^\top, \bar{\xi}^\top, 0]^\top, $$
        $$d\mathbb{E}[\hat{X}] = \left[(\hat{\mathcal{A}} + \hat{\mathcal{B}} \hat{\mathcal{M}}) \mathbb{E}[\hat{X}] + \hat{\mathcal{B}} \hat{\mathcal{V}} + \hat{\mathfrak{f}} \right] dt, \quad \mathbb{E}[\hat{X}(0)] = [\bar{\xi}_0^\top, \bar{\xi}^\top, 0].$$
    \end{mythm}

    Combining Theorem \ref{team followers thm2} and Theorem \ref{team leader thm2}, (\ref{optimal strategy of team followers}) and (\ref{optimal strategy of team leader}) constitute a decentralized Stackelberg-team equilibrium.

    \section{Numerical simulation}

    In this section, we provides numerical examples for Problem {\bf (PS)} to validate our conclusions, and since Problem {\bf (PG)} is analogous, we omit the details.
    Consider Problem {\bf (PS)} for 30 followers.
    We set $A_0=0.05$, $B_0=0.1$, $f_0=1$, $D_0=1$, $A=-0.05$, $B=0.05$, $f=1$, $D=1$, $\Gamma_0=1$, $\eta_0=1$, $Q_0=1$, $R_0=1$, $\Gamma=0.8$, $\Gamma_1=1$, $\eta=0.05$, $Q=1$, $R=0.1$. The time interval is $[0,10]$.
    The initial states of 30 followers are independently drawn from a uniform distribution on $[0,20]$, while the initial states of the leader is drawn from a uniform distribution on $[0,10]$.
    The curves of followers' Riccati equations $\hat{P}(\cdot)$, $\hat{K}(\cdot)$ and $\hat{\Pi}(\cdot)$, described by (\ref{Riccati-1 of team followers}), (\ref{Riccati-2 of team followers}) and (\ref{Riccati-1+2 of team followers}), are shown in Figure \ref{PKPi_plot}.
    \begin{figure}[htbp]
        \centering\includegraphics[width=12cm]{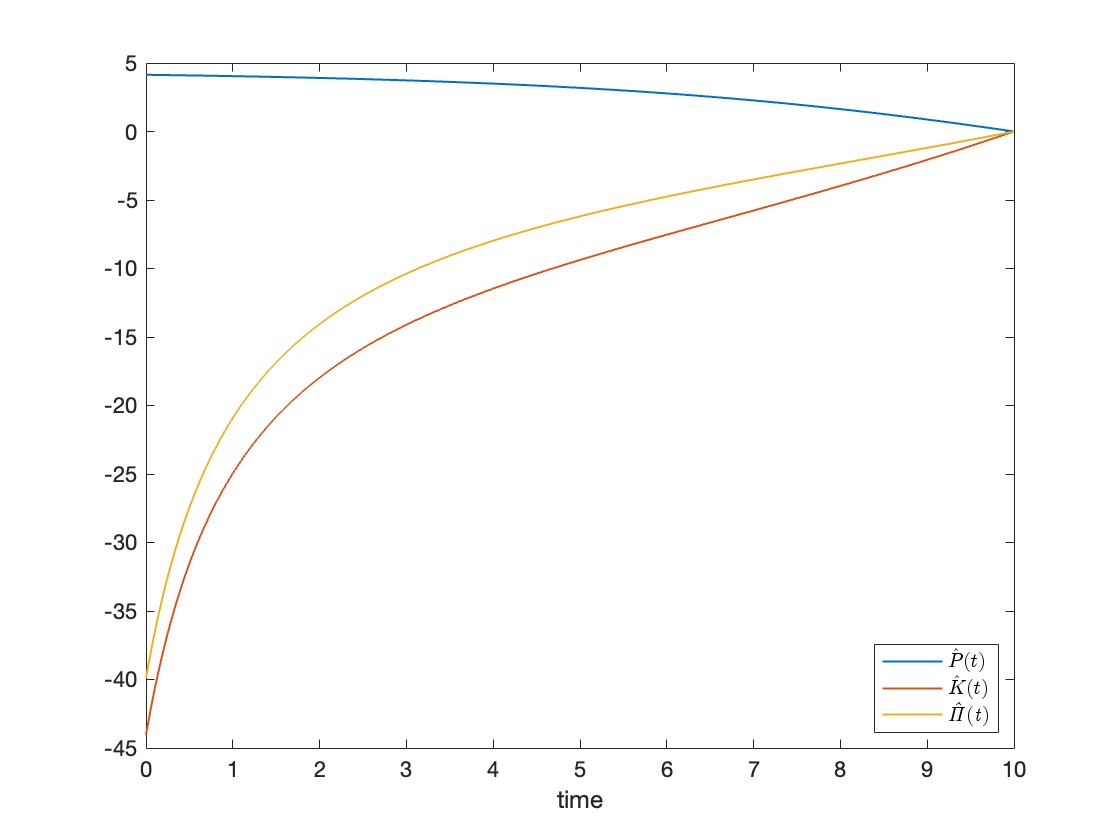}
        \caption{The curves of $\hat{P}(t)$, $\hat{K}(t)$ and $\hat{\Pi}(t)$}
        \label{PKPi_plot}
    \end{figure}
    The solution to the leader's Riccati equation (\ref{Riccati-1 of team leader}), $\hat{\mathcal{P}}(\cdot) \in \mathbb{R}^{3 \times 3}$, is depicted in Figure \ref{mathcal_P_plot}.
    \begin{figure}[htbp]
        \centering\includegraphics[width=12cm]{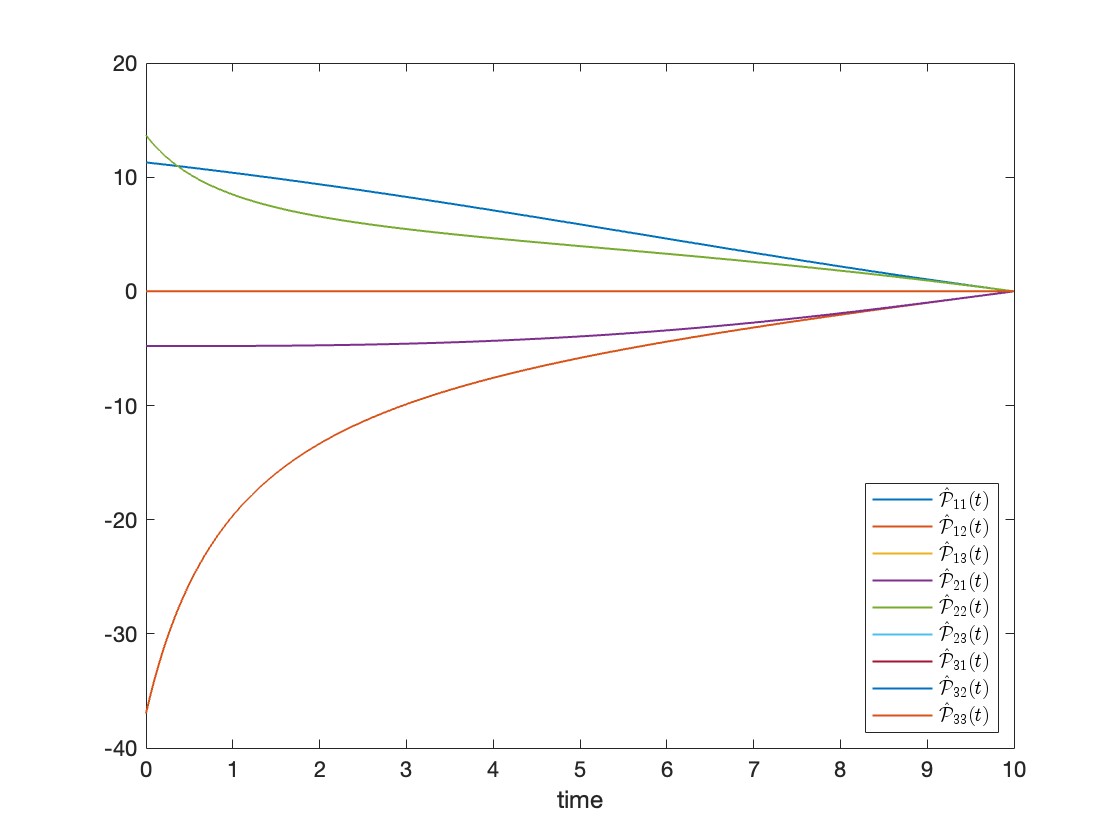}
        \caption{The curve of $\hat{\mathcal{P}}(\cdot)$}
        \label{mathcal_P_plot}
    \end{figure}
    Under the strategies (\ref{optimal strategy of team followers}) and (\ref{optimal strategy of team leader}), the trajectories of 30 followers and $\mathbb{E}[\hat{x}_i]$ are shown in Figure \ref{trajectory_plot}.
    In which the gradient surface represents $\mathbb{E}[\hat{x}_i]$, we can observe that the trajectories of 30 followers closely approximates $\mathbb{E}[\hat{x}_i]$.
    \begin{figure}[htbp]
        \centering\includegraphics[width=12cm]{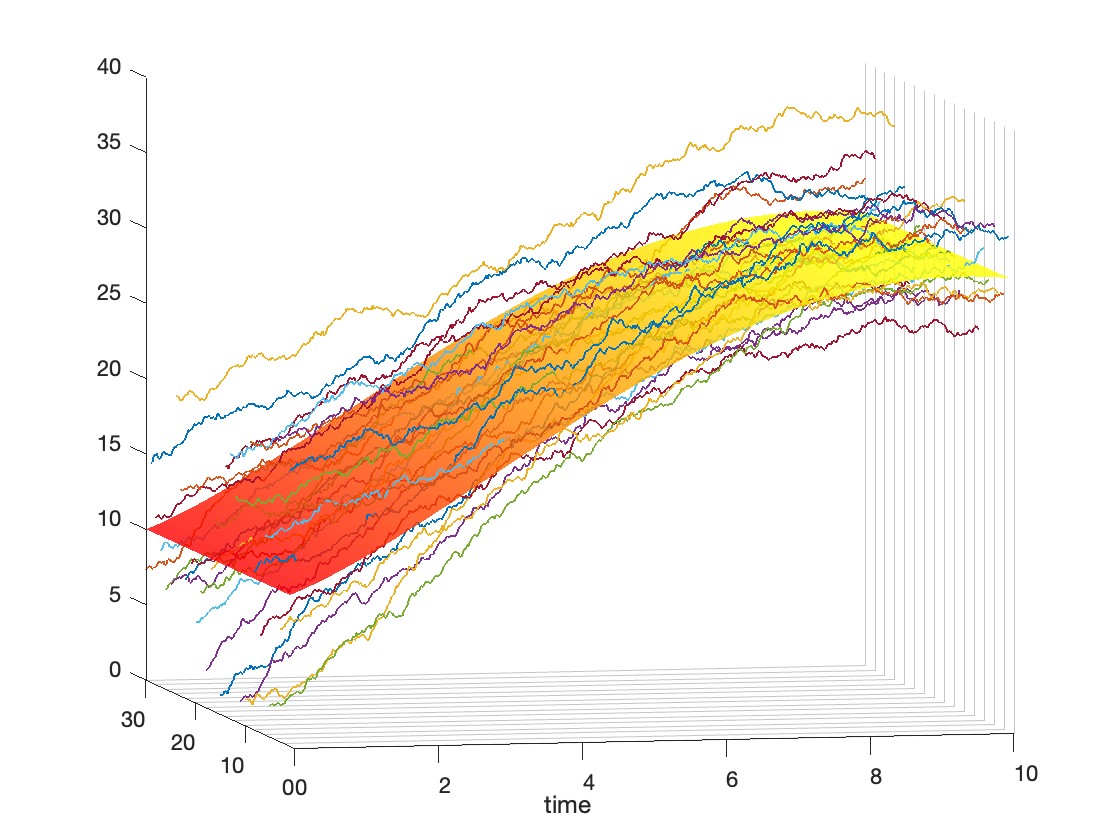}
        \caption{Trajectories of 30 followers and $\mathbb{E}[\hat{x}_i]$}
        \label{trajectory_plot}
    \end{figure}

    \section{Conclusions}

    In this paper, we have explored an LQ Stackelberg MF games and teams problem with arbitrary population sizes by a new de-aggregation method, and construct a set of exact decentralized Stackelberg-Nash or Stackelberg-team equilibrium strategies.
    In future research, this method could be applied to more complex scenarios, such as state equations with coupled MF terms, and state equations with multiplicative noise.

\end{document}